\documentclass[11pt]{article}

\usepackage{a4}
\usepackage{geometry}
\usepackage{amsfonts}
\usepackage{graphicx}
\usepackage{epstopdf}
\usepackage{algorithmic}
\usepackage{amsmath}
\usepackage{amsthm}
\usepackage{chngcntr}
\usepackage{hyperref}

\ifpdf
  \DeclareGraphicsExtensions{.eps,.pdf,.png,.jpg}
\else
  \DeclareGraphicsExtensions{.eps}
\fi

 \newtheorem{theorem}{theorem}[section]
 \newtheorem{corollary}[theorem]{Corollary}
 \newtheorem{lemma}[theorem]{Lemma}
 \theoremstyle{definition}
 \newtheorem{definition}[theorem]{Definition}
 \theoremstyle{remark}
 \newtheorem{remark}[theorem]{Remark}
 \numberwithin{equation}{section}

\title{Homogenization of Non-Local Navier-Stokes-Korteweg Equations for  Compressible Liquid--\-Vapour Flow in Porous Media\thanks{{\textbf{Funding. }}
		 The authors would thank the German Research Foundation (DFG) for funding this work within SFB 1313, Research Project C02.
  }}

\author{Christian Rohde
\thanks{Institute of Applied Analysis and Numerical Simulation,
		University of Stuttgart,
		Pfaffenwaldring 57,
		70569 Stuttgart, 
		Germany
  (\mbox{christian.rohde@mathematik.uni-stuttgart.de}, \mbox{lars.von-wolff@mathematik.uni-stuttgart.de}).
 }
\and Lars von Wolff\, \footnotemark[2]  
}

\usepackage{amsopn}

\usepackage{tikz}
\usepackage{cases}
\usetikzlibrary{arrows.meta}

\newcommand{\NN}{{\mathbb{N}}}         
\newcommand{\ZZ}{{\mathbb{Z}}}       
\newcommand{\RR}{{\mathbb{R}}}
\newcommand{\Ys}{{\mathcal{Y}}}

\newcommand{\diff}{\,d}											
\newcommand{\weakto}{\rightharpoonup}							
\newcommand{\eps}{{\varepsilon}}								
\newcommand{\set}[1]{\left\{#1\right\}}            				
\newcommand{\lnorm}[2]{\left\lVert #1 \right\rVert_{L^{#2}}}	
\newcommand{\norm}[1]{\left\lVert #1 \right\rVert} 				
\newcommand{\dbdot}{\mathbin{:}}
\DeclareMathOperator{\supp}{supp}			 					
\DeclareMathOperator{\dist}{dist}			 					
\let\div\anymacrowhichisnotdefine
\DeclareMathOperator{\div}{div}									

\hypersetup{
  pdftitle={Homogenization for a Navier--Stokes--Korteweg System},
  pdfauthor={C. Rohde, L. von Wolff},
  pdfkeywords={Compressible two--phase flow, Navier--Stokes--Korteweg equations,  homogenization, weak solutions}
  pdfsubject={76M50, 76N99, 76T10}
}

\begin{document}

\maketitle

\begin{abstract}\noindent We consider a nonlocal version of the 
	quasi--static  Navier--Stokes--Korteweg
	equations with a non--monotone pressure law.  This system governs the   low--Reynolds number  dynamics  of a compressible viscous  fluid that may 
	take either a liquid   or a vapour state.  For a porous  domain that 
	is perforated by cavities with diameter proportional to their mutual distance the homogenization limit is analyzed. 
	We extend the results for  compressible  one--phase flow with  polytropic pressure laws  
	 and prove 
	that the effective motion is governed by a nonlocal version of the Cahn--Hilliard equation.
	Crucial for the analysis is the convolution--like structure of the nonlocal capillarity term 
	that allows to equip the system with a  generalized convex free energy. Moreover, the capillarity term 
	accounts not only for the energetic interaction within the fluid but also for the interaction with a solid wall boundary.
\end{abstract}

\providecommand{\keywords}[1]{{\textit{Key words:}} #1\\ \\}
\providecommand{\class}[1]{{\textit{AMS subject classifications:}} #1}

\keywords{Compressible two--phase flow, Navier--Stokes--Korteweg equations,  homogenization, weak solutions}
\class{76M50, 76N99, 76T10}

\section{Introduction}

Deriving effective models for fluid flow through porous domains by homogenization  is an important issue to understand  many 
natural and technical processes.  Incompressible one--phase flow  governed by the Stokes or stationary Navier--Stokes equations has been 
analyzed by Allaire in   \cite{Allaire90a,Allaire90b}. Depending on the pore size/pore distance ratio he proved that the governing 
effective laws do either not change in type or correspond to  Brinkman--  or Darcy--type  laws.  
If the pore size scales as the pore distance  Mikeli\'c  considered the incompressible fully  time--dependent Navier--Stokes system
to derive a Darcy system as effective law \cite{Mikelic91}. These results have then been extended to other scalings in  e.g.~\cite{FeiNamNec16}.
For  the compressible   Navier--Stokes system  with a polytropic pressure law  we refer to  the work of Masmoudi \cite{m} and Feireisl\&Lu \cite{FeireislLu15}. 
In this case the homogenization limit 
gets    us    for    fixed pore size/pore distance ratio to a nonlinear parabolic evolution, i.e., the porous medium equation. We mention also the recent contribution \cite{LuSchwarzacher18} for a small--size pore regime. \\
If the flow system under consideration involves more than one fluid or a fluid in multiple states the possible homogenization scenarios 
feature a wider variety of effective laws, but much less rigorous results are known. For immiscible viscous two--phase flow 
in a thin domain Mikeli\'c\&Paoli identified the Buckley--Leverett equation as the effective  law \cite{MikPao97}.
In this paper we are interested 
in a homogeneous compressible fluid with viscosity that can occur in two states, say a liquid and a vapour one. Up to our knowledge  
no homogenization  results have been derived in this situation. In fact, the choice of the model itself on the pore scale is still
a matter of research and widely debated.  We consider here an instance of the compressible Navier--Stokes--Korteweg (NSK)
equations. To enable the liquid--vapour  phase transformation the constitutive law  to relate pressure to density 
is given  for the NSK systems by a non-monotone Van--der--Waals like function.  As a consequence the first--order part of the system 
is of mixed, hyperbolic-elliptic type.  The different instances of the NSK systems differ in the choice of the  capillarity term.
The standard variant  traces back to \cite{dunnserrin} (see also \cite{Anderson}) and relies on a local differential operator. With the resulting 
third--order capillarity term  and the elliptic--hyperbolic structure of the underlying Euler system     homogenization appears to be complicated. An alternative are lower-order but  nonlocal choices  as in \cite{Neusser15,rohde,rohde2}. These models do not only avoid the
third--order terms but allow the identification of a generalized monotone pressure function   
which is essential for  our method of proof. In passing we note that the generalized monotone pressure function does not reduce to a purely polytropic 
law as considered in \cite{FeireislLu15,m} such that further refinement of the method of proof are required.  The nonlocal capillarity term consists of two parts: one controls the  multiphase interaction within the fluid states 
while theother one governs the  exchange with the solid boundary. Both contributions give rise to an extended  free energy formulation based 
on classical  fluid--fluid and fluid--solid interaction potentials (see  e.g.~\cite{Fischer1998} or \cite{Pismen2001}).

The quasi--static model in the pore space will be introduced in Section  \ref{chapter_Model}, in particular we  outline 
the notion of a generalized pressure and the energetic structure of the system.  In Section \ref{chapter_Prelim} we  specify
the homogenization framework which relies on a fixed pore space/pore distance ratio. Theorem \ref{theorem} in Section \ref{chapter_Result} contains 
our main result. We show that the effective law is given by a nonlocal Cahn--Hilliard equation. Section \ref{chapter_Proof} is then devoted to the proof of 
Theorem \ref{theorem}. The proof relies on a combination of the techniques in \cite{FeireislNovotny,lions,m}. Furthermore 
we carefully exploit the regularizing  benefits  of the generalized pressure function and use properties of  the convolution structure of the 
nonlocal capillarity operator.

\section{A Nonlocal  Model for Two--Phase Flow\label{chapter_Model}}
We consider a diffuse interface model for  a homogeneous compressible fluid that can occur in a liquid and a vapour state.   Phase boundaries should 
be displayed as continuous transitions over a small distance that is controlled by a scaling parameter. Precisely, we focus on a  nonlocal  version  of the Navier--Stokes--Korteweg
(NSK) models from  \cite{rohde, rohde2}. \\ 
Our quasi--static fluid regime  covers  small Reynolds numbers, i.e., viscous forces dominate the  inertial forces. With a  non--dimensionalization as in 
\cite{diaz} we are then led to  the following  form of the NSK model.

For a bounded domain $X \subset \RR^N$, $N\in\set{2,3}$ and a time interval $(0,T)$ the density $\rho: (0,T)\times X \to \RR_{\geq 0}$ and the velocity $u: (0,T)\times X \to \RR^N$ obey the system
\begin{align}
\begin{split}
\label{model_universal}
\omega\partial_t \rho + \div(\rho u) &= 0  ,\\
-\mu \Delta u - \xi \nabla \div(u) + \nabla p(\rho) &= \gamma \rho \nabla \mathcal{D}_X[\rho]
\end{split}   \quad  \text{ in } (0,T)\times X 
\end{align}
and satisfy  for initial density $\rho_0: X \to \RR_{\geq 0}$ the 
initial and boundary conditions
\begin{align}
\label{model_bdry}
\begin{cases}
\rho (0,\cdot) = \rho_{ 0} & \text{in}\;\;X,\\
u = 0 & \text{on}\;\;(0,T) \times \partial X.
\end{cases}
\end{align}
Here $\omega>0$    is a small parameter that will be later put in relation to the homogenization
parameter. For the viscosity coefficients in \eqref{model_universal}  we assume  $\mu>0$ and $\xi\geq 0$. The capillarity 
constant $\gamma$ is assumed to be positive. Before we specify the  capillarity  operator $\mathcal{D}_X$  in \eqref{model_universal} let us discuss  pressure functions 
that allow a    
\begin{figure}[!bp]
 \centering
\begin{tikzpicture}[domain=-1:1.05, samples=100]
\begin{scope}[scale = 4]
\def\alphax{{0.5-0.5*1/sqrt(6)}}
\def\alphaxt{-1/sqrt(6)}
\def\alphay{{0.5+0.5*1/sqrt(6)}}
\def\alphayt{1/sqrt(6)}
\def\hscale{0.8}

\draw[-Stealth] (0,0)--(1.2,0) node[below]{$\rho$};
\draw[-Stealth] (0,0)--(0,1.0);
\draw (0,0) -- (0,-1pt) node[below]{0};
\draw (\alphax,0) -- (\alphax,-1pt) node[below]{$\alpha_1$};
\draw (\alphay,0) -- (\alphay,-1pt) node[below]{$\alpha_2$};

\draw[dotted] (\alphax,0) -- (\alphax, {((\alphaxt)^3-0.5*(\alphaxt) + 0.5)*\hscale});
\draw[dotted] (\alphay,0) -- (\alphay, {((\alphayt)^3-0.5*(\alphayt) + 0.5)*\hscale});
\begin{scope}[xscale = 0.5, yscale = \hscale]
\draw[shift = {(1,0.5)}] plot (\x, {(\x)^3-0.5*(\x)}) node[left]{$p(\rho)$};
\end{scope}
\end{scope}
\end{tikzpicture}\qquad
  \begin{tikzpicture}[domain=0.05:2.4, samples=100]
\begin{scope}[scale = 4]
\def\alphax{{0.5-0.5*1/sqrt(6)}}
\def\alphaxt{-1/sqrt(6)}
\def\alphaxtt{(1-1/sqrt(6))}
\def\alphay{{0.5+0.5*1/sqrt(6)}}
\def\alphayt{1/sqrt(6)}
\def\alphaytt{(1+1/sqrt(6))}
\def\hscale{3}

\draw[-Stealth] (0,0)--(1.2,0) node[below]{$\rho$};
\draw[-Stealth] (0,0)--(0,1.0);
\draw (0,0) -- (0,-1pt) node[below]{0};
\draw (\alphax,0) -- (\alphax,-1pt) node[below]{$\alpha_1$};
\draw (\alphay,0) -- (\alphay,-1pt) node[below]{$\alpha_2$};

\draw[dotted] (\alphax,0) -- (\alphax, {(2*\alphaxtt + 0.5*\alphaxtt^3 - 3*\alphaxtt^2 + 2.5*\alphaxtt*ln(\alphaxtt)+0.6)*\hscale});
\draw[dotted] (\alphay,0) -- (\alphay, {(2*\alphaytt + 0.5*\alphaytt^3 - 3*\alphaytt^2 + 2.5*\alphaytt*ln(\alphaytt)+0.6)*\hscale});
\begin{scope}[xscale = 0.5, yscale = \hscale]
\draw[shift = {(0,0.6)}] plot (\x, {2*\x + 0.5*\x^3 - 3*\x^2 + 2.5*\x*ln(\x)}) node[left]{$W(\rho)$};
\end{scope}
\end{scope}
\end{tikzpicture}
 \caption{Left: Example for a pressure function $p(\rho)$ in a two-phase setting. Right: The corresponding energy function $W(\rho)$}
  \label{Figure_W}
\end{figure}
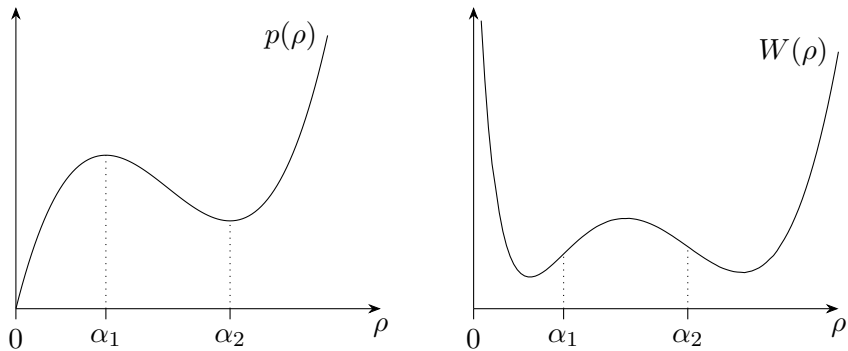
setting with two fluid states. The pressure function $p=p(\rho): [0,\infty) \to [0,\infty)$ is assumed to be  monotone increasing in some interval $[0,\alpha_1]$, monotone decreasing in $(\alpha_1,\alpha_2)$ and again monotone increasing in $[\alpha_2, \infty)$. For $\rho$ in these three intervals we call the {fluid state} {{vapour}}, {{spinodal}} or {{liquid}}, respectively. An illustration of such a pressure function is given in
Figure \ref{Figure_W}. In contrast to a one--phase fluid with monotone increasing pressure the first order flux  in \eqref{model_universal} 
is not purely hyperbolic but  hyperbolic--elliptic. \\ 
We associate an energy function $W: [0,\infty) \to [0,\infty)$ with the pressure $p$ through the condition
\begin{align}
\label{model_W}
p'(\rho) = \rho W''(\rho).
\end{align}
Therefore changes in the monotonicity of $p$ translate to changes in the convexity of $W$. In our two--phase setting this results in a double--well structure, also illustrated in Figure \ref{Figure_W}.\\
By rewriting \eqref{model_universal} we will later work with the \textit{generalized  pressure function} given by 
\begin{align}
\label{model_P}
P(\rho) := p(\rho)+\frac{\gamma}{2} \rho^2.
\end{align}
Our main result (Theorem \ref{theorem}) requires some properties of $P$, and therefore implicitly imposes conditions on $p$ and $\gamma$.
\begin{definition}
	\label{def_pressure}
	A function $P: [0,\infty) \to \RR$ is called an \textbf{admissible generalized pressure function} if $P$ satisfies
	\begin{enumerate}
		\item \label{def_pressure_i1}$P \in C^2([0,\infty))$,
		\item \label{def_pressure_i2}$P(0)=0$,
		\item \label{def_pressure_i3}$P'(r)\geq \alpha$ and $P''(r)\geq \alpha$ for some constant $\alpha > 0$,
		\item \label{def_pressure_i4}$f(r) := r\cdot P^{-1}(r)$ is convex, where $P^{-1}(r)$ denotes the inverse function of $P$. This is equivalent to $P \cdot P''\cdot (P')^{-2} \leq 2$,
		\item \label{def_pressure_i5}Let there be constants $\beta\geq 2$ and $c>0$ such that $\frac{P'(r)}{r^{\beta-1}} \to c$ as $r\to \infty$. Then also $\frac{P(r)}{r^{\beta}} \to c/\beta$.
	\end{enumerate}
\end{definition}
It is easy to see that this allows for a two--phase setting, because  $p$ can be a non--monotone function if $\gamma>0$  is chosen big enough. 
On the other hand Definition \ref{def_pressure} allows us to choose $\gamma=0$ and a monotone $p$ such that $P=p$ is admissible. In this case we include the  single--state setting. Theorem \ref{theorem} will then coincide with the result of  Masmoudi  in \cite{m}, at least  for a polytropic pressure 
law.    

\begin{remark}
	Definition \ref{def_pressure}  accounts also for   Van--der--Waals  pressure laws, that is
	\begin{align*}
	p(\rho) = \frac{RT_\ast \rho}{b-\rho} - a\rho^2 ,
	\end{align*}
	with positive constants $a$, $b$, $R$, $T_\ast$ and the critical Temperature $T_\ast$ small enough, so that the pressure function is non--monotone.
	The only technical difference is that the last condition in Definition \ref{def_pressure} $p$ does not cover $\rho \to \infty$ at finite density, 
	but letting it diverge polynomially for $\rho \to  \infty$.
\end{remark}

Using the admissible generalized pressure instead of the original pressure will become important in our analysis as it leads to a
convex generalized energy function. The convexifying 
quadratic term $\gamma \frac{\rho^2}{2}$ in \eqref{model_P}  will be separated from the capillarity term $\mathcal{D}_X[\rho]$
which we describe in the next step. The operator $\mathcal{D}_X$  is supposed to model capillary forces between different fluid phases as well as between the  fluid and the solid occupying $\RR^N\setminus X$. As mentioned above we prefer among many possible choices 
a nonlocal set--up (see \cite{rohde2}) which requires the following definition of an interaction kernel. 
\begin{definition}
	\label{def_kernel}
	A smooth, compactly supported function  $\phi : \RR^N \to \RR$  is called an \textbf{interaction kernel} if it satisfies  
	\begin{align}
	\label{model_kernel}
	\phi \geq 0,\quad \phi(0)>0,\quad  \phi(x)=\phi(-x), \quad \int_{\RR^N} \phi(x)\diff x = 1.
	\end{align}	
\end{definition}

For an interaction kernel $\phi$ and a constant wall density $\rho_s > 0$ we let the  operator $\mathcal{D}_X$  acting on a density field $\rho( t,\cdot)$  be given by 
\begin{align*}
\mathcal{D}_X[\rho(t,\cdot)](x)   
=\int_{X} \phi(x-y)(\rho(t,y)-\rho(t,x)) \diff y  
 + \int_{\RR^N\setminus X} \phi(x-y)(\rho_s-\rho(t,x)) \diff y .
\end{align*}
Using the notation
\begin{align}
\label{abbrev_convolution}
(\phi \ast_X \rho)(t,x) := \int_{X} \phi(x-y)\rho(t,y) \diff y + \int_{\RR^N\setminus X} \phi(x-y)\rho_s \diff y
\end{align}
we can write the operator $\mathcal{D}_X$ in the compact form
\begin{align}
\label{model_Dnonlocal}
\mathcal{D}_X[\rho(t,\cdot)](x) &=(\phi \ast_X \rho)(t,x)-\rho(t,x).
\end{align}
The model \eqref{model_universal} is now complete. Most notably it obeys the second law of thermodynamics: 
it is easy to see that   classical solutions $(\rho,u)$ of the nonlocal NSK model \eqref{model_universal},  \eqref{model_bdry} with \eqref{model_Dnonlocal} satisfy the energy balance
\begin{align*}
&\frac{d}{dt} \bigg(   \frac{\gamma \omega}{4}  \int_{X}\int_X  \phi(x-y) \left(\rho(t,x)-\rho(t,y)\right)^2 \diff y \diff x   \\
&\qquad \quad + \frac{\gamma \omega}{2}  \int_X\int_{\RR^N\setminus X} \phi(x-y) \left(\rho(t,x)-\rho_s\right)^2  \diff y \diff x
+ \omega \int_X W(\rho(t,x))  \diff x \bigg) \\[2ex]
& \qquad = -\int_{X} \mu  |\nabla u |^2  +  \xi \big(\div(u)\big)^2  \diff x.
\end{align*}
The free energy splits up into three parts consisting of a fluid--fluid interaction energy, a fluid--solid interaction energy and the homogeneous bulk 
energy.  The fluid--solid interaction energy is constructed in the same way as the fluid-fluid interaction energy using a constant wall density outside of 
$X$. Writing the operator $\mathcal{D}_X$ in the form of \eqref{model_Dnonlocal} requires the energies to share the same interaction kernel $\phi$. For the derivation of specific non--local models for  fluid--solid interactions we refer to \cite{Fischer1998, Pismen2001}. 

\begin{remark}
	Note that in the nonlocal  model \eqref{model_universal}, \eqref{model_Dnonlocal}  the solid--fluid interaction is not realized by a contact--line
	boundary condition  as in local two--phase models. Anyhow, a  further boundary condition  would render   the lower--order model  \eqref{model_universal} to be overdetermined. Taking into account the energy balance above    we expect  solutions of \eqref{model_universal}, \eqref{model_Dnonlocal} to be of wetting type, i.e,    to develop a possibly very narrow liquid layer with $\rho$ approaching $ \rho_s$ in the vicinity of a solid wall. A detailed investigation of the fluid states close to the wall can  be found \cite{Pismen2001}.  
\end{remark}

\section{The Porous Domain and Homogenization Scalings\label{chapter_Prelim}}
We summarize first basic  notations for  function spaces that are needed in the sequel. 
\subsection{Notations}
\label{chapter_not}
For matrices $A=(a_{ij}), B=(b_{ij}) \in \RR^{N\times N}$ we write $A:B := \sum_{i,j} a_{ij}b_{ij}$.

For a scalar function $f$ the gradient is denoted by $\nabla f$. For a vector valued function $g$ we write $Dg$ for the Jacobian, $\div g$ for the divergence and $\Delta g$ for the Laplace operator applied component--wise. These operators are only applied to spatial coordinates.

We will denote the set of infinitely differentiable functions on a domain $X$ by $C^\infty(X)$. $C^\infty_C(X)$ consists of all functions in $C^\infty(X)$ with compact support. 

For $r\in[1,\infty]$ we denote the Lebesgue spaces on a domain $X$ by $L^r(X)$. If it is not ambiguous we will denote the space $L^r(X)$ just by $L^r$, e.g. $\norm{f}_{L^r}$ stands for the $L^r$-norm on the domain of the function $f:X\to\RR$. 

$W^{k,r}(X)$ will denote the Sobolev space of order $k\in \NN$ and we write $H^k(X)$ for $W^{k,2}(X)$. $W^{k,r}_0(X)$ is the closure of $C^\infty_C(X)$ in $W^{k,r}(X)$. The dual space of $H^1_0(X)$ will be called $H^{-1}(X)$.

Most of the time we will deal with functions defined on some space--time domain $[0,T]\times X$. For a Banach space $E$ let us denote by $C_T(E)$ the space of continuous functions on $[0,T]$ with codomain $E$ and by $L^r_T(E)$ the Lebesgue space on $[0,T]$ with codomain $E$. We will mainly use the spaces $L^s_T(W^{k,r}(X))$. Note also the isomorphism $L^r_T(L^r(X)) \cong L^r([0,T]\times X)$.

The notation $L^r(X)^n$ is used for vector--valued functions with $n\in\NN$ components, where each component is an element of $L^r(X)$. Similarly we write $L^r(X)^{n\times n}$ for matrix valued functions. In both cases we might shorten the notation to $L^r(X)$.

In Section \ref{chapter_domain} we will introduce the homogenization parameter $\eps >0$. Let us outline that we 
use in the sequel a  constant $C>0$  as a  generic constant that might depend on the data of our problem but not on $\eps$. Furthermore we introduce some weighted spaces.
\begin{definition}
	\label{def_spacesum}
	For $\eps>0$ the \textbf{$\eps$-weighed sum} $E+\eps F$ of two Banach spaces $E$, $F$ with $E \subseteq F$ is given by the space $F$ endowed with the norm
	\begin{align*}
	\norm{f} :&= \inf \set{\norm{f_1}_E + \norm{f_2}_F \mid f = f_1 + \eps f_2, f_1 \in E , f_2 \in F} \\
	&= \inf \set{\norm{f_1}_E + \eps^{-1} \norm{f_2}_F \mid f = f_1 + f_2, f_1 \in E , f_2 \in F}.
	\end{align*}
	The \textbf{$\eps$-weighed intersection} $(\eps E) \cap F$ is given by the space $E$ endowed with the norm
	\begin{align*}
	\norm{f} :&= \eps \norm{f}_E + \norm{f}_F .
	\end{align*}
\end{definition}

In a Banach space $E$ we denote strong convergence of a sequence $\set{f_k}\subset E$ to $f \in E$ by $f_k \to f$, weak convergence by $f_k \weakto f$ and weak-$\ast$ convergence by $f_k \overset{\ast}\weakto f$.


\subsection{The Porous Domain}
\label{chapter_domain}
Let $\Omega$ be a  bounded domain in  $\RR^N$  with smooth boundary for  $N=2$ or $N=3$.  We denote the {unit cell} by $\Ys := (0,1)^N\subset \RR^N$ and want the {solid grain}  domain $\Ys_s$ to be  a closed subset of $\Ys$ with smooth boundary. Denote its  $N$-dimensional Lebesgue measure by $|\Ys_s|>0$. Then the {fluid part} is given by $\Ys_f := \Ys\setminus \Ys_s$, see Figure \ref{Figure_Omega_eps} for an example. We define the {porosity} $\theta := |\Ys_f|$ and deduce $0<\theta<1$.

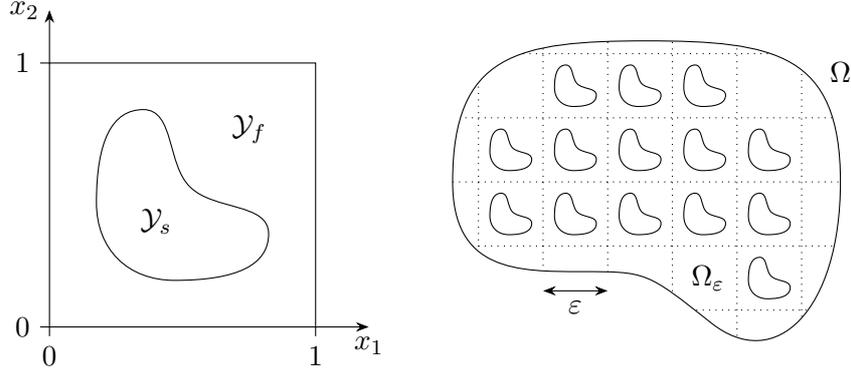
\begin{figure}[!bpt]
\centering
  \begin{tikzpicture}
\begin{scope}[scale = 3.5]

\draw[-Stealth] (1,0)--(1.2,0) node[below]{$x_1$};
\draw[-Stealth] (0,1)--(0,1.2) node[left]{$x_2$};
\draw (0,0) -- (0,-1pt) node[below]{0};
\draw (1,0) -- (1,-1pt) node[below]{1};
\draw (0,0) -- (-1pt,0) node[left]{0};
\draw (0,1) -- (-1pt,1) node[left]{1};

\draw (0,0) rectangle (1,1);

\begin{scope}[scale = 0.058823] 
\draw [shift = {(3,8)}]
	  (0,0) .. controls (0,4) and (1,6) .. (3,6)
			.. controls (5,6) and (4.5,2.4) .. (6,1)
			.. controls (7.5,-0.5) and (11,0) .. (11,-2)
			.. controls (11,-4) and (9,-5) .. (5,-5)
			.. controls (2,-5) and (0,-3) .. cycle;
\end{scope}

\draw (0.4,0.4) node{$\Ys_s$};
\draw (0.75,0.75) node{$\Ys_f$};

\end{scope}
\end{tikzpicture}\qquad
\begin{tikzpicture}
\begin{scope}[scale = 0.85]

\draw (2,1.3)[Stealth-Stealth] -- node[below] {$\eps$} (3,1.3);	
\draw (4.55,1.5) node{$\Omega_\eps$};
\draw (6.6,4.4) node[above] {$\Omega$};

\draw[rotate=180,shift = {(-6.6,-5)}, clip]
	  (0,2) .. controls (0,4) and (1,5) .. (2,4.2)
			.. controls (3,3.4) and (3,3.4) .. (4,3.4)
			.. controls (5,3.4) and (6,3.4) .. (6,2)
			.. controls (6,0) and (5,-0.2) .. (3,-0.2)
			.. controls (1,-0.2) and (0,0) .. cycle;
			
\draw[dotted] (0.5,0.5) grid[step=1] (6.7,5.3);
		
\foreach \position in {(1,2),(1,3),(2,2),(2,3),(2,4),
					   (3,2),(3,3),(3,4),(4,2),(4,3),(4,4),
					   (5,1),(5,2),(5,3)}
{
	\begin{scope}[shift = {\position}]
	\begin{scope}[scale = 0.058823] 
	\draw [shift = {(3,8)}]
		  (0,0) .. controls (0,4) and (1,6) .. (3,6)
				.. controls (5,6) and (4.5,2.4) .. (6,1)
				.. controls (7.5,-0.5) and (11,0) .. (11,-2)
				.. controls (11,-4) and (9,-5) .. (5,-5)
				.. controls (2,-5) and (0,-3) .. cycle;
	\end{scope}
	\end{scope}
}
\end{scope}
\end{tikzpicture}
  
  \caption{Left: Example of the unit cell $\Ys$ with solid part $\Ys_s$ and fluid grain $\Ys_f$ for $N=2$. Right: Construction of $\Omega_\eps$ from the domain $\Omega$ including periodically continued and re-scaled solid grains.}
  \label{Figure_Omega_eps}
\end{figure}

To describe a porous domain we introduce the homogenization parameter $\eps >0$.  We  copy $\Ys_f$ periodically (i.e.,  duplicating $\Ys_f$ shifted by $k$ for each $k\in \ZZ^N$), then rescale by a factor of $\eps$ and intersect with the  domain $\Omega$. For technical reasons we do not remove solid grain that would intersect with $\partial \Omega$. Let us define the set of indices corresponding to cells in the interior of $\Omega$ as
\[
K_\eps := \set{k\in\ZZ^N \mid \eps (\Ys + k) \subset \Omega}.
\]
With this we can define the \textit{porous domain} $\Omega_\eps$ as
\[
\Omega_\eps := \Omega \setminus \bigcup_{k\in K_\eps} \eps(\Ys_s + k).
\]
Note that $\Omega_\eps \subset \Omega$ has a smooth boundary. An illustration of $\Omega_\eps$ is given in Figure \ref{Figure_Omega_eps}.
We observe in particular that for this construction the ratio  between the pores' diameters (scaled fluid parts) and  the distance between scaled fluid parts remains constant with respect to variation of $\eps$.


We sometimes need to work on all cells in the interior of $\Omega_\eps$. For this let
\begin{align*}
\Omega_K :=  \left(\bigcup_{k\in K_\eps} \eps(\overline{\Ys} + k)\right)^o  \quad \text{and} \quad \Omega_{K,\eps} :=  \left(\bigcup_{k\in K_\eps} \eps(\overline{\Ys_f} + k)\right)^o = \Omega_K \cap \Omega_\eps,
\end{align*}
where we denoted the closure by $\overline{\,\cdot\,}$ and the open interior by $\cdot ^o$ Lastly we will need a set with smooth boundary surrounding $\Ys_s$. For this let us first fix a smooth open region $\Ys_r$ with $\Ys_s \subset \Ys_r \subset \Ys$. We set $\Ys_{r\setminus s} := \Ys_r\setminus \Ys_s$.

\subsection{Basic Definitions for Homogenization}
\label{chapter_def}
We want to extend functions defined on $\Omega_\eps$ to the whole of $\Omega$. This will in particular enable us to compare  for $\eps_1,\eps_2 >0$ functions defined on different sets $\Omega_{\eps_1}$ and $\Omega_{\eps_2}$. Let us define two extension operators:
\begin{definition}[Extension operators]
	\label{def_extension}
	For a function $\phi \in L^1(\Omega_\eps)$ we define the \textbf{zero extension} $\tilde \phi \in L^1(\Omega)$ by
	\[
	\tilde \phi = \begin{cases}
	\phi & \text{in}\;\; \Omega_\eps , \\
	0 & \text{in}\;\; \Omega\setminus\Omega_\eps ,
	\end{cases}
	\]
	and the \textbf{mean value extension} $\hat \phi \in L^1(\Omega)$ by
	\[
	\hat \phi = \begin{cases}
	\phi & \text{in}\;\; \Omega_\eps ,\\
	\dfrac{1}{\eps^N |\Ys_{r\setminus s}|} \displaystyle\int_{\eps(\Ys_{r\setminus s}+k)} \phi(x) \diff x & \text{in}\;\; \eps(\Ys_s+k),\;\; k\in K_\eps .
	\end{cases}
	\]
\end{definition}
\noindent
A relation between weak limits of mean value extensions and zero extensions is given by Lemma 1.3 from \cite{m}:
\begin{lemma}
	\label{lemma_extension}
	For  $\eps>0$  let $g_\eps \in L^1(\Omega_\eps)$ and $g\in L^1(\Omega) $. Then, the following two assertions are equivalent in the  
	limit $\eps \to 0$.
	\begin{enumerate}
		\item $\hat g_\eps \weakto g \quad \text{in}\;\; L^1(\Omega)$,
		\item $\tilde g_\eps \weakto \theta g \quad \text{in}\;\; L^1(\Omega)$.
	\end{enumerate}
\end{lemma}
\noindent
Furthermore, we refer to Lemma \ref{lemmma_restriction} for the construction of a restriction operator dual to the mean value extension.


We need  to define a permeability matrix, see \cite{m, tartar}.  Wellposedness and  regularity results from e.g.~Galdi \cite{galdi}) guarentee that 
we find for $1\leq i \leq N$ unique solutions $v_i \in C^\infty_{per}(\Ys_f)^N$ and $q_i \in C^\infty_{per}(\Ys_f)/\RR$ to the Stokes problem
\begin{align} 
\label{eqn_pdevi}
\begin{cases}
-\Delta v_i + \nabla q_i = e_i & \text{in}\;\; \Ys_f ,\\
\div v_i = 0 & \text{in}\;\; \Ys_f ,\\
v_i = 0  &\text{on}\;\; \partial\Ys_s .
\end{cases}
\end{align}
Here the space $C^\infty_{per}(\Ys_f)$ consists of functions $f\in C^\infty(\Ys_f)$, such that the periodic continuation 
\begin{align*}
g:\bigcup_{k\in \ZZ^N} (\overline{\Ys_f} + k) \to \RR,\quad  g(x-k) = f(x) \quad \text{for} \quad x\in\overline{\Ys_f}, k\in K_\eps
\end{align*}
is infinitely often differentiable. 

Let $\tilde v_i$ be the zero extension of $v_i$ to $\Ys$. Let $A(x)$, $x\in\Ys$ be the matrix with columns $\tilde v_i(x)$, $i=1,\ldots,N$. Then the \textit{permeability matrix} $\bar A \in \RR^{N\times N}$ is defined as the average of $A$:
\begin{align}
\label{eqn_barA}
\bar A := \int_\Ys A(x) \diff x.
\end{align}
Finally we 
define functions 
\begin{align*}
v^\eps_i\in W^{1,\infty}(\Omega)^N\quad \text{and} \quad q^\eps_i\in W^{1,\infty}(\Omega_{K,\eps})
\end{align*}
by rescaling and periodic continuation:
\begin{align}
\begin{split}
\label{definition_vq}
v^\eps_i(\eps (x-k)) = \tilde v(x) \quad &\text{for}\quad x\in\Ys, k \in \ZZ^N , \\
q^\eps_i(\eps (x-k)) = q(x) \quad &\text{for}\quad x\in\Ys_f, k\in K_\eps .
\end{split}
\end{align}
We get from the rescaling directly  the  uniform estimates 
\begin{align}
\label{bounds_vi}
\norm{v^\eps_i}_{L^\infty(\Omega)} + \eps \norm{D v^\eps_i}_{L^\infty(\Omega)}\leq C ,\\
\label{bounds_qi}
\norm{q^\eps_i}_{L^\infty(\Omega_{K,\eps})} + \eps \norm{\nabla q^\eps_i}_{L^\infty(\Omega_{K,\eps})}\leq C.
\end{align}
\section{The Main Result\label{chapter_Result}}
In this section we state our main result on the homogenization limit  on a sequence of porous domains $\Omega_\eps$ for $\eps \to 0$. 
For the ease of notation we  write for the convolution defined in \eqref{abbrev_convolution} 
\[
\phi \ast_\eps \rho := \phi \ast_{\Omega_\eps} \rho \mbox{ and } \phi \ast_0 \rho := \phi \ast_\Omega \rho.
\] 
In the same spirit  we abbreviate $\mathcal{D}_\eps := \mathcal{D}_{\Omega_\eps} $  for the capillarity operator. 
For the proof of  the main theorem an important tool is the  statement of  Lemma \ref{lemma_conv1} on the convergence 
of a sequence of convolutions on $\Omega_\eps$.

Let us  now rewrite the nonlocal NSK model \eqref{model_universal} on the domain $X = \Omega_\eps$ with
the operator $\mathcal{D}_{\eps}$ and the choice $\omega = \eps^2$. We search for a 
density $\rho_\eps: (0,T)\times \Omega_\eps \to \RR_{\geq 0}$ and a velocity $u_\eps: (0,T)\times \Omega_\eps \to \RR^N$ that 
obey the system
\begin{align}
\begin{split}
\label{model_universaleps}
\eps^2\partial_t \rho_\eps + \div(\rho_\eps u_\eps) &= 0  ,\\
-\mu \Delta u_\eps - \xi \nabla \div(u_\eps) + \nabla p(\rho_\eps) &= \gamma \rho \nabla \mathcal{D}_\eps[\rho_\eps]
\end{split}  \quad \text{ in } (0,T)\times \Omega_\eps
\end{align}
and satisfy  for initial density $\rho_{0,\eps}: \Omega_\eps \to \RR_{\geq 0}$ the 
initial and boundary conditions
\begin{align}
\label{model_bdryeps}
\begin{cases}
\rho_\eps (t=0) = \rho_{ 0,\eps} & \text{in}\;\;\Omega_\eps,\\
u_\eps = 0 & \text{on}\;\;(0,T) \times \partial \Omega_\eps.
\end{cases}
\end{align}
The initial density is supposed to  satisfy  $W(\rho_{0,\eps}) \in L^1(\Omega_\eps)$ with a uniform bound on $\set{\norm{W(\rho_{0,\eps})}_{L^1(\Omega_\eps)}}_{\eps>0}$, with the energy function $W$ given by \eqref{model_W}. 
Additionally, we require a uniform bound on $\set{\norm{\rho_{0,\eps}}_{L^2(\Omega_\eps)}}_{\eps>0}$.
Furthermore  let the sequence $\set{\hat{\rho}_{\eps 0}}_{\eps > 0}$ be weakly convergent in $L^1(\Omega)$ with the weak limit denoted by $\rho_0 \in L^1(\Omega)$.

For  the nonlocal NSK model \eqref{model_universaleps}, \eqref{model_bdryeps} we require a weak formulation in the following sense.
\begin{definition}[Weak solution to \eqref{model_universaleps}, \eqref{model_bdryeps}]
	Some functions
	\begin{align*} 
	\rho_\eps\in C_T(L^2(\Omega_\eps)) \quad \text{and}\quad u_\eps\in L^2_T(H^1_0(\Omega_\eps))^N
	\end{align*} 
	with $P(\rho_\eps)\in L^2_T(L^2(\Omega_\eps))$ are called a \textbf{weak solution} to the model \eqref{model_universal}, \eqref{model_bdry} if $\rho_\eps \geq 0$ a.e. and if for any $ f\in C^\infty_C([0,\infty))$
	\begin{align}
	\begin{split}
	\label{weak_conti}
	&\int_0^T \int_{\Omega_\eps} \eps^2 f(\rho_\eps) \partial_t \psi +f(\rho_\eps) u_\eps \nabla \psi  - (\div u_\eps)[f'(\rho_\eps)\rho_\eps - f(\rho_\eps)]\psi \diff x \diff t \\
	&\qquad= -\int_{\Omega_\eps} \eps^2 f(\rho_{0,\eps}) \psi(t=0) \diff x,
	\end{split}\\
	\begin{split}
	\label{weak_mom}
	&\int_0^T \int_{\Omega_\eps} \mu D u_\eps \dbdot D v +\xi \div u_\eps \div v - p(\rho_\eps) \div v \diff x \diff t \\
	&\qquad = \int_0^T \int_{\Omega_\eps} \gamma \rho_\eps \nabla\left(\phi \ast_\eps \rho_\eps\right) v + \frac{\gamma}{2} \rho_\eps^2 \div v \diff x \diff t
	\end{split}
	\end{align}
	hold for all test functions $\psi \in C^\infty_C((-\infty,T) \times \Omega_\eps)$, $v \in L^2_T(H^1_0(\Omega_\eps))^N$.
\end{definition}
\begin{remark}
	For an admissible generalized pressure function $P$ the condition $P(\rho_\eps)\in L^2_T(L^2(\Omega_\eps))$ implies $\rho_\eps\in L^2_T(L^4(\Omega_\eps))$, and with this all integrals in the weak formulation are finite.
\end{remark}

\begin{remark}
	This weak formulation is used in the semi--stationary model described by Lions \cite{lions} in Chapter 8.2, see in particular Theorem 8.6. In view of Remark 8.14 and Chapter 7.5  in \cite{lions}  one can generalize the  existence result for this semi--stationary model to the case with an admissible generalized pressure function $P$. In fact, this approach has been used in \cite{Haspot} to derive a global existence theorem for weak 
	solutions of the nonlocal NSK equations.
\end{remark}

We will prove that in the limit $\eps \to 0$ the  evolution of the limit density $\rho: [0,T]\times \Omega \to \RR_{\ge 0} $   will be 
governed by the  nonlocal Cahn--Hilliard problem      
\begin{align}
\label{Result_CH}
\theta\partial_t\rho + \frac{1}{\mu} \div\left[\rho \bar A \left(\gamma \theta \rho \nabla(\phi\ast_0\rho - \rho) - \nabla \left(p(\rho)+\frac{\gamma (1-\theta)}{2}\rho^2 \right)\right)\right] = 0 
\end{align}
in $(0,T)\times\Omega$, with initial and boundary conditions
\begin{align}
\label{Result_CH_IVBP}
\begin{cases}
\rho(t=0) = \rho_0 &\text{in}\;\;\Omega ,\\
\rho \bar A \left(\gamma \theta \rho \nabla(\phi\ast_0\rho - \rho) - \nabla \left(p(\rho)+\frac{\gamma (1-\theta)}{2}\rho^2 \right)\right)\cdot n = 0 &\text{on}\;\; (0,T)\times\partial\Omega .
\end{cases}
\end{align}
Here $n\in \RR^N$ denotes the normal vector on $\partial \Omega$. We recall that $\theta$ is the porosity and $\bar A\in \RR^{N\times N}$ is the permeability matrix of the porous medium (see Section \ref{chapter_def} for definitions). 
Nonlocal Cahn--Hilliard problems have been introduced in  \cite{Giacomin1998} as models for phase separation dynamics. Another asymptotic
regime for Korteweg fluids  that is governed by the  Cahn-Hilliard equation can be found in \cite{Tzavaras17}.

\begin{definition}[Weak solution to \eqref{Result_CH}, \eqref{Result_CH_IVBP}]
	A function $\rho \in L^2_T(H^1(\Omega))$ with
	\begin{align*}
	\bar A \left(\gamma \theta \rho \nabla(\phi\ast_0\rho - \rho) - \nabla \left(p(\rho)+\frac{\gamma (1-\theta)}{2}\rho^2 \right)\right) \in L^2_T(L^2(\Omega))
	\end{align*}
	is called a \textbf{weak solution} to \eqref{Result_CH}, \eqref{Result_CH_IVBP} if $\rho\geq 0$ a.e. and 
	\begin{align*}
	0&=\int_0^T \int_\Omega \theta \rho \partial_t \psi \diff x \diff t +\int_\Omega \theta \rho_{0} \psi(t=0) \diff x\\
	&\qquad + \frac{1}{\mu}\int_0^T \int_\Omega \rho \bar A \left(\gamma \theta \rho \nabla(\phi\ast_0\rho - \rho) - \nabla \left(p(\rho)+\frac{\gamma (1-\theta)}{2}\rho^2 \right)\right) \nabla \psi \diff x \diff t
	\end{align*}
	hold for all test functions $\psi \in C^\infty_C((-\infty,T)\times \RR^N)$.
\end{definition}
\noindent
Now we are ready to state our main theorem.
\begin{theorem}
	\label{theorem}
	Consider \eqref{model_universaleps}, \eqref{model_bdryeps} with an admissible generalized  pressure function $P(\rho) = p(\rho)+\frac{\gamma}{2} \rho^2$, with operator $\mathcal{D}_\eps$ and an interaction kernel $\phi$ satisfying \eqref{model_kernel}. For a sequence $\eps\to 0$ let $\rho_\eps\in C_T(L^2)$, $u_\eps\in L^2_T(H^1_0)^N$  be weak solutions to \eqref{model_universaleps}, \eqref{model_bdryeps}.
	
	Then there exist functions $u \in L^2_T(L^2(\Omega))^N$, $\rho \in L^2_T(H^1(\Omega))$ such that we have for a subsequence of the extensions 
	$\set{\tilde u_\eps}_{\eps > 0}$, $\set{\hat \rho_\eps}_{\eps > 0}$
	\begin{align}
	\frac{\tilde u_\eps}{\eps^2} \weakto u \quad  &\text{in}\quad L^2_T(L^2)^N ,\\
	\quad \hat \rho_\eps \to \rho \quad  &\text{in}\quad  L^2_T(L^2) .
	\end{align}
	Furthermore $\rho$ weakly solves the nonlocal Cahn--Hilliard problem \eqref{Result_CH}, \eqref{Result_CH_IVBP}. 
	
	We have on $\set{(t,x)\in (0,T)\times \Omega\mid \rho(t,x) >0}$ the Darcy--like relation
	\begin{align}
	\label{eqn_result_2x}
	u=\frac{1}{\mu} \bar A \left(\gamma \theta \rho \nabla(\phi\ast_0\rho - \rho) - \nabla \left(p(\rho)+\frac{\gamma (1-\theta)}{2}\rho^2 \right)\right).
	\end{align}
\end{theorem}
The next section is completely devoted to the proof of Theorem \ref{theorem}.

\section{Proof of Theorem 4.5} \label{chapter_Proof}
Throughout the section we suppose that all assumptions and notations as stated in Theorem \ref{theorem} are valid.
\subsection{A Priori Estimates}
\label{chapter_apriori}
We will use the following Poincar\'e inequality, see also Lemma 1.5 from \cite{m}:
\begin{lemma}[Poincar\'e inequality in $\Omega_\eps$]
	\label{lemma_poincare}
	There exists a constant $C_p>0$ which depends only on $\Ys_s$ such that for all $u\in W^{1,p}_0(\Omega_\eps)$ and for all $\eps >0$ we have the estimate
	\begin{align*}
	\norm{u}_{L^p(\Omega_\eps)} \leq C_p \eps \norm{D u}_{L^p(\Omega_\eps)} .
	\end{align*}
\end{lemma}

A straightforward generalization of Lemma 2.3  in \cite{m}  on bounded domains to the 
entire space  is the following:
\begin{lemma}
	\label{lemma_masmoudi1} With a slight misuse of notation we consider  the zero extensions $\tilde \rho_\eps$, $\tilde u_\eps$ and $ \tilde \rho_{0,\eps}$ to $\RR^N$. These extensions satisfy
	\begin{align*}
	\int_0^T \int_{\RR^N} \eps^2\tilde{\rho}_\eps \partial_t \psi +\tilde{\rho}_\eps \tilde{u}_\eps \nabla \psi \diff x \diff t = -\int_{\RR^N} \eps^2\tilde \rho_{0,\eps} \psi(t=0) \diff x\quad 
	\end{align*}
	for all  $\psi \in C^\infty_C((-\infty,T) \times \RR^N)$. 
\end{lemma} 
Using test functions $\psi\in C^\infty_C((-\infty,T) \times \RR^N)$ that are constant in $\Omega_\eps$ (for each time $t$) we get
\[
\int_0^T\int_{\Omega_\eps} \eps^2 \rho_\eps \partial_t \psi \diff x \diff t = -\int_{\Omega_\eps} \eps^2 \rho_{0,\eps} \psi(t=0) \diff x,
\]
and thus we have conservation of mass, that is
\begin{align}
\int_{\Omega_\eps} \rho_\eps(t) \diff x = \int_{\Omega_\eps} \rho_{0,\eps} \diff x.
\end{align}

We find the following a priori estimates.
\begin{lemma}[A priori estimates in $\Omega_\eps$]
	\label{lemma_bounds1}
	There exist uniform bounds on the respective norms of $\frac{u_\eps}{\eps^2} \in L^2_T(L^2(\Omega_\eps))$, $\frac{u_\eps}{\eps} \in L^2_T(H^1(\Omega_\eps))$, $W(\rho_\eps)\in L^\infty_T(L^1(\Omega_\eps))$ and $\rho_\eps\in L^\infty_T(L^2(\Omega_\eps))$.
\end{lemma}

\begin{proof}
	Let us test the weak formulation \eqref{weak_mom} with $u_\eps$ times the indicator function of $(0,\tau)$ for a fixed $\tau \in (0,T)$, that is
	\begin{align}
	\label{eqn_weaktest}
	\begin{split}
	&\int_0^\tau \int_{\Omega_\eps} \mu D u_\eps \dbdot D u_\eps +\xi (\div u_\eps)^2 - p(\rho_\eps) \div u_\eps - \frac{\gamma}{2} \rho_\eps^2 \div u_\eps \diff x \diff t \\
	&\qquad = \int_0^\tau \int_{\Omega_\eps} \gamma \rho_\eps \nabla\left(\phi \ast_\eps \rho_\eps\right) u_\eps  \diff x \diff t .
	\end{split}
	\end{align}
	With $p(\rho)=\rho W'(\rho) - W(\rho)$ the third term of \eqref{eqn_weaktest}
	calculates to
	\begin{align} 
	\begin{split}
	\label{eqn_rewrite2}
	&\int_0^\tau\int_{\Omega_\eps} p(\rho_\eps) \div u_\eps \diff x \diff t
	= \int_0^\tau\int_{\Omega_\eps}  [W'(\rho_\eps)\rho_\eps -W(\rho_\eps)] \div u_\eps \diff x \diff t \\
	&\qquad = -\int_{\Omega_\eps} \eps^2 W(\rho_\eps(\tau)) \diff x+\int_{\Omega_\eps} \eps^2 W(\rho_{0,\eps}) \diff x .
	\end{split}
	\end{align}
	To get to the second line we have used the weak formulation \eqref{weak_conti} whith $f$ approximating $W$, and $\psi$ approximating the indicator function of $(0,\tau)$. By arguing analogously to the proof of Lemma 2.3 of Masmoudi \cite{m} we can choose $\psi$ constant in space. For the fourth term of \eqref{eqn_weaktest} we can use the same method, but we let $f(\rho_\eps)$ approximate $\frac{\gamma}{2} \rho_\eps^2$
	\begin{align} 
	\begin{split}
	\label{eqn_rewrite3}
	&\int_0^\tau\int_{\Omega_\eps} \frac{\gamma}{2} \rho_\eps^2 \div u_\eps \diff x \diff t
	= \int_0^\tau\int_{\Omega_\eps} \left[(\gamma\rho_\eps) \rho_\eps - \frac{\gamma}{2} \rho_\eps^2\right] \div u_\eps \diff x \diff t \\
	&\qquad = -\int_{\Omega_\eps}  \frac{\gamma\eps^2}{2} \rho_\eps(\tau)^2 \diff x+\int_{\Omega_\eps}  \frac{\gamma\eps^2}{2} \rho_{0,\eps}^2 \diff x.
	\end{split} 
	\end{align}
	We can estimate the right hand side of \eqref{eqn_weaktest} for a fixed time $t$ by
	\begin{align} 
	\begin{split}
	\label{eqn_rewrite4}
	&\left|\int_{\Omega_\eps} \gamma\rho_\eps \nabla(\phi \ast_\eps \rho_\eps)u_\eps\diff x\right|
	\leq \gamma \lnorm{\rho_\eps(t)}{2}\lnorm{u_\eps(t)}{2}\lnorm{\nabla(\phi \ast_\eps \rho_\eps(t))}{\infty} \\
	&\qquad\leq C \lnorm{\rho_\eps(t)}{2}\lnorm{u_\eps(t)}{2} \left(\lnorm{\rho_\eps(t)}{1}+\rho_s\right) \\
	&\qquad\leq C C_p \eps \lnorm{\rho_\eps(t)}{2}\lnorm{D u_\eps(t)}{2} \left(\lnorm{\rho_{0,\eps}}{1}+\rho_s\right) \\
	&\qquad\leq \frac{\mu}{2} \lnorm{D u_\eps(t)}{2}^2 + \frac{C^2 C_p^2 \eps^2}{2\mu} \lnorm{\rho_\eps(t)}{2}^2 \left(\lnorm{\rho_{0,\eps}}{1}+\rho_s\right)^2.
	\end{split} 
	\end{align}
	Here we have used Lemma \ref{lemma_conv1} to get to the second line, the Poincar\'e inequality (Lemma \ref{lemma_poincare}) to get to the third line, and Young's inequality for products to get the last line. Overall we have from \eqref{eqn_weaktest}, \eqref{eqn_rewrite2}, \eqref{eqn_rewrite3} and \eqref{eqn_rewrite4} and the uniform bounds on $\norm{W(\rho_{0,\eps})}_{L^1(\Omega_\eps)}$ and $\norm{\rho_{0,\eps}}_{L^2(\Omega_\eps)}$
	\begin{align*}
	&\frac{\mu}{2} \|D u_\eps \|_{L^2_\tau(L^2)}^2 + \int_{\Omega_\eps} \eps^2 W(\rho_\eps(\tau)) \diff x +  \frac{\gamma \eps^2}{2} \lnorm{\rho_\eps(\tau)}{2}^2\\
	&\qquad \leq \eps^2
	C\left( 1+ \int_0^\tau \lnorm{\rho_\eps(t)}{2}^2 \diff t\right)
	\end{align*}
	Using Gronwall's inequality we get for $\tau \in (0,T)$
	\begin{align*}
	&\frac{\mu}{2} \left\|\frac{D u_\eps}{\eps} \right\|_{L^2_\tau(L^2)}^2 + \int_{\Omega_\eps} W(\rho_\eps(\tau)) \diff x +  \frac{\gamma}{2} \lnorm{\rho_\eps(\tau)}{2}^2 \leq C .
	\end{align*}
	This gives us most of the estimates. For the estimate on $\frac{u_\eps}{\eps^2}$ we use the Poincar\'e inequality (Lemma \ref{lemma_poincare}) and the bound on $\frac{D u_\eps}{\eps}$.
\end{proof}

Next, we establish some bounds for functions extended to the limit domain $\Omega$. For this let us first define
\begin{align*}
\check \rho_\eps := P^{-1}(\hat P(\rho_\eps)). 
\end{align*} 
The main point here is that we get some spatial regularity for $\hat P(\rho_\eps)$ and $\check \rho_\eps$.
\begin{lemma}[A priori estimates in $\Omega$]
	\label{lemma_bounds2}
	There exist uniform bounds on the respective norms of $\hat P(\rho_\eps) \in L^\infty_T(L^1(\Omega))$, $\hat P(\rho_\eps) \in L^2_T(H^1(\Omega)) + \eps L^2_T(L^2(\Omega))$, $\check \rho_\eps \in L^2_T(H^1(\Omega)) + \eps L^2_T(L^2(\Omega))$ and $\check \rho_\eps \in L^{2\beta}_T(L^{2\beta}(\Omega))$.
\end{lemma}

\begin{proof}
	We will use the constant $\beta\geq 2$ from Definition \ref{def_pressure} for the admissible generalized  pressure function $P$. In the case $\beta = 2$, that is $\lim_{r\to\infty} \frac{P(r)}{r^2} < \infty$, we directly get a uniform bound on $\norm{P(\rho_\eps)}_{L^\infty_T(L^1)}$ from our uniform bound on $\norm{\rho_\eps}_{L^\infty_T(L^2)}$. Otherwise we have $\beta > 2$ and can use Lemma \ref{lemma_pressure} together with our uniform bound on $\norm{W(\rho_\eps)}_{L^\infty_T(L^1)}$ to get to the same conclusion. With Definition \ref{def_extension} we have a uniform bound on $\|\hat P(\rho_\eps)\|_{L^\infty_T(L^1(\Omega))}$.
	
	Next, we argue that $\hat P(\rho)$ has some spatial regularity. With the restriction operator $R_\eps$ constructed in Lemma \ref{lemmma_restriction} we calculate for $v\in L^2_T(H^1_0(\Omega))^N$
	\begin{align*}
	&\int_0^T \langle \nabla \hat P(\rho_\eps), v \rangle_{H^{-1},H^1(\Omega)} \diff t
	:= -\int_0^T \int_\Omega \hat P(\rho_\eps) \div v \diff x \diff t \\
	&\qquad =  -\int_0^T \int_{\Omega_\eps} P(\rho_\eps) \div R_\eps v \diff x\diff t 
	= \int_0^T \langle \nabla P(\rho_\eps), R_\eps v \rangle_{H^{-1},H^1(\Omega_\eps)} \diff t .
	\end{align*}
	We also have with Lemma \ref{lemmma_restriction} and the bounds from Lemma \ref{lemma_bounds1}
	\begin{align*}
	&\left| \int_0^T \langle \nabla P(\rho_\eps), R_\eps v \rangle_{H^{-1},H^1(\Omega_\eps)}  \diff t \right| \\
	&\qquad = \left| \int_0^T \int_{\Omega_\eps} -\mu D u_\eps \dbdot D R_\eps v - \xi \div u_\eps \div R_\eps v+ \gamma \rho_\eps \nabla(\phi \ast_\eps \rho_\eps) R_\eps v \diff x \diff t\right| \\
	&\qquad \leq (\mu+\xi) \eps \norm{\frac{D u_\eps}{\eps}}_{L^2_T(L^2)} \norm{D R_\eps v}_{L^2_T(L^2)} \\
	&\qquad\qquad + \gamma \norm{\rho_\eps}_{L^2_T(L^2)} \norm{\nabla(\phi \ast_\eps \rho_\eps)}_{L^\infty_T(L^\infty)} \norm{R_\eps v}_{L^2_T(L^2)}\\
	&\qquad \leq C \left( \norm{v}_{L^2_T(L^2(\Omega))} + \eps \norm{D v}_{L^2_T(L^2(\Omega))} \right) .
	\end{align*}
	This means that $\nabla \hat P(\rho_\eps)$ is a bounded linear operator on $(\eps L^2_T(H^1_0(\Omega)))\cap L^2_T(L^2(\Omega))$, see Definition \ref{def_spacesum}. By identifying $L^2$ with its dual and using results on the sum and intersection of Banach spaces (see e.g. \cite{lions}, Appendix E, and more general \cite{krein}, Theorem 3.1)
	we have
	\begin{align*}
	\big[ (\eps L^2_T(H^1_0(\Omega)))\cap L^2_T(L^2(\Omega))\big] ^\ast = L^2_T(L^2(\Omega))+ \eps L^2_T(H^{-1}(\Omega)).
	\end{align*}
	We therefore have a uniform bound on the norms of
	\begin{align*}
	\nabla \hat P(\rho_\eps) \in L^2(L^2(\Omega))+ \eps L^2(H^{-1}(\Omega)),
	\end{align*}
	and can write $\nabla \hat P(\rho_\eps) =  F_\eps + \eps G_\eps$ with $\set{F_\eps} \subset L^2_T(L^2)$ and $\set{G_\eps} \subset L^2_T(H^{-1})$ being uniformly bounded. Let the operator $S$ be defined by $Sf = p$ where $v\in H^1_0(\Omega)$ and $q\in L^2(\Omega)/\RR$ solve the Stokes problem
	\begin{align*}
	\begin{cases}
	-\Delta v + \nabla q = f & \text{in}\;\; \Omega ,\\
	\div v = 0 & \text{in}\;\; \Omega.
	\end{cases}
	\end{align*}
	By regularity results (see e.g. \cite{galdi}) the operators $S: H^{-1}(\Omega)\to L^2(\Omega)/\RR$ and $S: L^2(\Omega)\to H^1(\Omega)/\RR$ are bounded. With this we get
	\begin{align*}
	\hat P(\rho_\eps) + c_\eps= S(\nabla \hat P(\rho_\eps)) = S(\nabla F_\eps) + \eps S(\nabla G_\eps),
	\end{align*}
	where the additive constant $c_\eps$ can still depend on $\eps$. We have shown the uniform bound on the norms of  
	\begin{align*}
	\big(\hat P(\rho_\eps)+c_\eps \big) \in L^2_T(H^1(\Omega)) + \eps L^2_T(L^2(\Omega)).
	\end{align*}
	Our bound on $\|\hat P(\rho_\eps)\|_{L^\infty_T(L^1(\Omega))}$ implies a bound on $\|\hat P(\rho_\eps)\|_{L^2_T(L^1)}$. With this argument 
	$c_\eps$ can be  bounded as follows.
	\[
	\norm{c_\eps}_{L^2_T(L^1)} \leq \norm{\hat P(\rho_\eps)}_{L^2_T(L^1)} + \norm{\hat P(\rho_\eps)+c_\eps}_{L^2_T(L^1)} 
	\]
	Therefore the norms of $\hat P(\rho_\eps) \in L^2_T(H^1) + \eps L^2_T(L^2)$ are uniformly bounded.

	Note that by requirements \ref{def_pressure_i1} and \ref{def_pressure_i3} of Definition \ref{def_pressure} the inverse $P^{-1}$ is smooth with $0\leq (P^{-1})'\leq \frac{1}{\alpha}$. Using the regularity of $\hat P(\rho_\eps)$ we can write $\hat P(\rho_\eps) =  f_\eps + \eps g_\eps$ where $\set{f_\eps} \subset L^2_T(H^1)$ and $\set{g_\eps} \subset L^2_T(L^2)$ are uniformly bounded. Then
	\begin{align*}
	\norm{P^{-1}(f_\eps + \eps g_\eps) - P^{-1}(f_\eps)}_{L^2_T(L^2)} \leq \frac{1}{\alpha} \eps \norm{g_\eps}_{L^2_T(L^2)}.
	\end{align*}
	As $P^{-1}(f_\eps)$ is uniformly bounded in $L^2_T(H^1)$ we now have a uniform bound on the norms of
	\begin{align}
	\check \rho_\eps := P^{-1}(\hat P(\rho_\eps)) \in L^2_T(H^1) + \eps L^2_T(L^2).
	\end{align}
	On the other hand we can also use requirement \ref{def_pressure_i5} of Definition \ref{def_pressure} to see that we have a uniform bound on the norm of $\check \rho_\eps \in L^{2\beta}_T(L^{2\beta})$.
	
\end{proof}
\subsection{Strong Convergence}
\label{chapter_sc}
With estimates from the previous section we have the following convergences.
\begin{corollary}
	\label{cor_wc}
	After extracting subsequences of $\eps \to 0$ we have the existence of the following weak limits:
	\begin{align}
	\label{wc_u}
	&\exists \;u \in L^2_T(L^2)^N &&\text{with} && \frac{\tilde u_\eps}{\eps^2} \weakto u \; \text{in}\; L^2_T(L^2)^N ,\\
	&\exists \;\rho \in L^2_T(L^2) && \text{with} && \hat \rho_\eps \weakto \rho \; \text{in}\; L^2_T(L^2) 
	\label{wc_hatrho}\\
	&&& \text{and} && \tilde \rho_\eps \weakto \theta \rho \; \text{in}\; L^2_T(L^2) ,
	\label{wc_tilderho}\\
	&\exists \;\bar P \in L^2_T(H^1) && \text{with} && \hat P(\rho_\eps) \weakto \bar P \; \text{in}\; L^2_T(L^2) ,
	\label{wc_P}\\
	&\exists \;q \in L^2_T(H^1)\cap L^{2\beta}_T(L^{2\beta})  && \text{with} && \check \rho_\eps \weakto q \; \text{in}\; L^{2\beta}_T(L^{2\beta}) .
	\label{wc_q}
	\end{align}
\end{corollary}
\begin{proof}
	With the a priori estimates of Lemma \ref{lemma_bounds1} and Lemma \ref{lemma_bounds2} we can use the weak compactness of the unit sphere in the respective reflexive function spaces. So after repeatedly extracting subsequences we have the existence of the asserted weak limits. Note that \eqref{wc_P}, \eqref{wc_q} do not provide strong convergence, as we have sufficient regularity in space but not in time. Lastly \eqref{wc_tilderho} follows from \eqref{wc_hatrho} by Lemma \ref{lemma_extension}.
\end{proof}

We can use these weak convergences to establish strong convergence:
\begin{lemma}
	\label{lemma_sc}
	For the weak limits of Corollary \ref{cor_wc} it holds $\rho = q = P^{-1}(\bar P)$. Furthermore we have
	\begin{enumerate}
		\item \label{lemma_sc_i1}$\hat \rho_\eps \to \rho$ in $L^2_T(L^2)$,
		\item \label{lemma_sc_i2}$\check \rho_\eps \to \rho$ in $L^r_T(L^r)$ for any $r \in [1, 2\beta)$,
		\item \label{lemma_sc_i3}$\nabla(\phi\ast_\eps \rho_\eps) \to \theta\nabla(\phi \ast_0 \rho)$ in $L^r_T(L^r)$ for any $r \in [1,\infty)$,
		\item \label{lemma_sc_i4}$\hat P(\rho_\eps) \to \bar P$ in $L^r_T(L^r)$ for any $r\in [1, 2)$.
	\end{enumerate}
\end{lemma}

\begin{proof}
	Let us first show $\rho = q = P^{-1}(\bar P)$. To do so we use the convexity of $P$ following from requirement \ref{def_pressure_i3} of Definition \ref{def_pressure}. Then by Jensen's inequality
	\begin{align*}
	P(\check \rho_\eps) = \hat P(\rho_\eps) \geq P(\hat \rho_\eps) .
	\end{align*}
	Because $P$ is increasing by requirement \ref{def_pressure_i3} of Definition \ref{def_pressure} we have 
	\begin{align}
	\label{eqn_rhohatcheck}
	\check \rho_\eps \geq \hat \rho_\epsilon .
	\end{align}
	Passing to the weak limit  yields
	\begin{align}
	\label{qgeqrho}
	q\geq \rho .
	\end{align}
	We now use Lemma \ref{lemma_weak_leq} with the convex function $P(|\cdot|)$, the sequence $\check \rho_\eps$ and the domain $(0,T)\times\Omega$. In the lemma we set $a = 2\beta$ and $b=\beta$. The lemma implies that the weak limit of $\hat P(\rho_\eps) = P(\check\rho_\eps)$ is bigger or equal to $P(q)$. In other words $\bar P \geq P(q)$. As $P^{-1}$ is increasing by requirement \ref{def_pressure_i3} of Definition \ref{def_pressure} we conclude
	\begin{align}
	\label{barpgeqq}
	P^{-1}(\bar P) \geq q .
	\end{align}	
	We note that Lemma \ref{lemma_masmoudi1} implies together with our a priori bounds  that $\partial_t \tilde \rho_\eps$ is uniformly bound in $L^1_T(W^{-1,1}(\Omega))$. Together with the higher spatial regularity of $\hat P(\rho_\eps)$ this is enough to deduce another weak limit: Using \eqref{wc_tilderho}, \eqref{wc_P} and Lemma 5.1 of \cite{lions} we get
	\begin{align*}
	\tilde \rho_\eps \hat P(\rho_\eps) \weakto \theta \rho \bar P,
	\end{align*}
	weakly in $L^1_T(L^1)$. By Lemma \ref{lemma_extension} we have
	\begin{align*}
	\widehat{\rho_\eps P(\rho_\eps)} \weakto \rho \bar P .
	\end{align*}
	Recall that by requirement \ref{def_pressure_i4} of Definition \ref{def_pressure} the function $f(r) = r \cdot P^{-1}(r)$ is convex. Because of this we can use Jensen's inequality
	\begin{align}
	\label{eqn_fhat}
	\widehat{\rho_\eps P(\rho_\eps)} = \widehat{f(P(\rho_\eps))} \geq f\left(\hat P(\rho_\eps)\right) .
	\end{align}
	Using Definition \ref{def_pressure} we can easily see that $0\leq f'(r) \leq 2\frac{r}{\alpha}$. Therefore $f$ can be bound by a quadratic function and with Lemma \ref{lemma_bounds2} we have a uniform bound on $\{f(\hat P(\rho_\eps))\} \subset L^2_T(L^1)$. Let us again pass to a subsequence of $\eps \to 0$ to guarantee the existence of a weak limit of $f(\hat P(\rho_\eps))$ in $L^2_T(L^1)$. Now we use Lemma \ref{lemma_weak_leq} with the convex function $f(|\cdot|)$ (and corresponding $b = 2$) and the sequence $\hat P(\rho_\eps)$ (thus $a= 2$ in the lemma). We conclude that the weak limit of $f(\hat P(\rho_\eps))$ is greater or equal to $f(\bar P)$. Together with \eqref{eqn_fhat} we get in the limit $\eps \to 0$
	\begin{align*}
	\rho \bar P \geq f(\bar P) = \bar P \cdot P^{-1}(\bar P) .
	\end{align*}
	As $P^{-1}(0)=0$ by requirement \ref{def_pressure_i2} of Definition \ref{def_pressure}, we have $\rho \geq P^{-1}(\bar P)$. Now together with \eqref{qgeqrho} and \eqref{barpgeqq} we conclude $\rho = q = P^{-1}(\bar P)$.

	To get strong convergence, we use again Lemma 5.1 of \cite{lions}, this time with $\tilde \rho_\eps$ and $\check \rho_\eps$. We conclude $\tilde \rho_\eps \check \rho_\eps \weakto \theta \rho q$. Note that $\check \rho_\eps = \tilde \rho_\eps$ whenever $\tilde \rho_\eps\neq 0$. Therefore by Lemma \ref{lemma_extension}
	\begin{align*}
	\widehat{\rho_\eps^2} \weakto \rho q = \rho^2 .
	\end{align*}
	By Jensen's inequality $(\hat \rho_\eps)^2 \leq \widehat{\rho_\eps^2}$ and therefore the previous statement implies norm--convergence $\norm{\hat \rho_\eps}_{L^2_T(L^2)} \to \norm{\rho}_{L^2_T(L^2)}$. Together with the weak convergence we deduce assertion \ref{lemma_sc_i1}, i.e. $\hat \rho_\eps \to \rho$ in $L^2_T(L^2)$.
	
	Next, recall from \eqref{eqn_rhohatcheck} that we have $0\leq \hat \rho_\eps \leq \check \rho_\eps$. Thus
	\begin{align*}
	\norm{\check \rho_\eps - \hat \rho_\eps}_{L^1_T(L^1)} = \int_{(0,T)\times\Omega} (\check \rho_\eps - \hat \rho_\eps) \diff x \diff t \to \int_{(0,T)\times\Omega} (q - \rho) \diff x \diff t = 0 .
	\end{align*}
	Together with the strong convergence of $\hat \rho_\eps$ this implies $\check \rho_\eps \to \rho$ in $L^1_T(L^1)$. By Lemma \ref{lemma_bounds2} the extended density $\check \rho_\eps$ is uniformly bounded in $L^{2\beta}_T(L^{2\beta})$, so interpolation between different $L^p$-norms gives assertion \ref{lemma_sc_i2}.
	
	We will also need the convergence of the term $\nabla(\phi\ast_\eps \rho_\eps)$. For any $1\leq r < \infty$ we get with Lemma \ref{lemma_conv1}
	\begin{align*}
	\begin{split}
	&\norm{\nabla(\phi\ast_\eps \rho_\eps) - \theta\nabla(\phi \ast_0 \rho)}_{L^r_T(L^r)} \\
	&\qquad\leq \norm{\nabla(\phi\ast_\eps \rho_\eps) - \nabla(\phi \ast_\eps \rho)}_{L^r_T(L^r)}  + 
	\norm{\nabla(\phi\ast_\eps \rho) - \theta\nabla(\phi \ast_0 \rho)}_{L^r_T(L^r)}\\
	&\qquad\leq C \norm{\nabla\phi}_{L^\infty}\norm{\rho_\eps-\rho}_{L^1_T(L^1(\Omega))}+\norm{\nabla(\phi\ast_\eps \rho) - \theta\nabla(\phi \ast_0 \rho)}_{L^r_T(L^r)} \\
	&\qquad\to  0 .
	\end{split}
	\end{align*}
	We have shown assertion \ref{lemma_sc_i3}.
	
	Now we can consider the convergence of $\hat P(\rho_\eps)$. For this we make a similar argument as in Lemma \ref{lemma_weak_leq}. From requirement \ref{def_pressure_i5} of Definition \ref{def_pressure} we can deduce that there are some constants $C_1$, $C_2$ such that $|P'(r)| \leq C_1 + C_2 |r|^{\beta-1}$ for all $r\geq 0$. We then calculate
	\begin{align*}
	&\norm{\hat P(\rho_\eps)-\bar P}_{L^1_T(L^1)} = \norm{P(\check \rho_\eps)-P(\rho)}_{L^1_T(L^1)} \\
	&\qquad\leq \int_{(0,T)\times\Omega} \left| \check \rho_\eps(x) - \rho(x)\right| \max_{y\in [\check \rho_\eps(x), \rho(x)]} |P'(y)| \diff x \diff t \\
	&\qquad\leq \int_{(0,T)\times\Omega} \left| \check \rho_\eps(x) - \rho(x)\right| \left(C_1 + C_2 \max(\check \rho_\eps(x),\rho(x))^{\beta-1}\right) \diff x \diff t \\
	&\qquad\leq C \norm{\check \rho_\eps - \rho}_{L^2_T(L^2)} \left(1 + \norm{\check \rho_\eps^{\beta-1}}_{L^2_T(L^2)} + \norm{\rho^{\beta-1}}_{L^2_T(L^2)}\right) .
	\end{align*}
	The first term tends to zero while the second term is uniformly bounded. We conclude $\hat P(\rho_\eps) \to \bar P$ in $L^1_T(L^1)$. Using Lemma \ref{lemma_bounds2} and the interpolation between $L^p$-norms we get assertion \ref{lemma_sc_i4}.
\end{proof}

\subsection{The Limit System}
To obtain a Cahn--Hilliard system for $\rho$ and $u$ we basically want to use $\rho v^\eps_k$ as a test function in the momentum equation \eqref{weak_mom}, and $\rho u_\eps$ as a test function in the equation satisfied by $v^\eps_k$ (see \eqref{eqn_pdevi} and \eqref{definition_vq}). Note that in \eqref{definition_vq} we did define $q^\eps_k$ not in $\Omega_\eps$ but only in the smaller domain $\Omega_{K,\eps}$. 

To handle the boundary, we multiply our test functions by some $\psi \in C^\infty_C((0,T)\times \Omega)$. We find a compact set $M\subset \Omega$ with $\supp(\psi(t))\subset M$ for all $t\in (0,T)$. Then $\dist(M,\partial \Omega)>0$ and $M \subset \Omega_{K,\eps}$ for $\eps$ small enough.

Furthermore we need a regularization of $\hat P(\rho_\eps)$: Let $\chi$ be a standard mollifier function 
and let $\chi_\eta(x) = \eta^{-N}\chi(x/\eta)$ for $\eta > 0$. For $\eta < \dist(M,\partial \Omega)$ we define 
\begin{align*}
P_{\eps,\eta} := \chi_\eta \ast \hat P(\rho_\eps),\; P_{\eta} := \chi_\eta \ast P(\rho) \in C^\infty(M).
\end{align*}
Because of Lemma \ref{lemma_bounds2} we can write $\hat P(\rho_\eps) =  f_\eps + \eps g_\eps$ where $f_\eps \in L^2_T(H^1)$ and $g_\eps \in L^2_T(L^2)$ are uniformly bounded. Using the regularity of $f_\eps$ we get with Lemma \ref{lemma_mol}
\begin{align}
\begin{split}
\label{bounds_hatP-Peta}
\norm{\hat P(\rho_\eps) - P_{\eps,\eta}}_{L^2_T(L^2(M))} &\leq \norm{f_\eps -\chi_\eta \ast f_\eps}_{L^2_T(L^2(M))} + 2\eps\norm{g_\eps}_{L^2_T(L^2(\Omega))} \\
&\leq C \eta \norm{f_\eps}_{L^2_T(H^1(\Omega))} + 2\eps\norm{g_\eps}_{L^2_T(L^2(\Omega))} \\
&\leq C(\eta + \eps) .
\end{split}
\end{align}
Recall that from \ref{lemma_sc} we have $\hat P(\rho_\eps) \to \hat P(\rho)$ in $L^r_T(L^r)$ for $r\in [1, 2)$.
In the limit $\eps \to 0$ we get with Young's convolution inequality
\begin{align}
\begin{split}
\label{sc_Peta}
\norm{\nabla P_{\eps,\eta} - \nabla P_{\eta}}_{L^r_T(L^r(M))} &= \norm{(\nabla \chi_\eta) \ast (\hat P(\rho_\eps) - \hat P(\rho))}_{L^r_T(L^r(M))} \\
&\leq \norm{\nabla \chi_\eta}_{L^1}  \norm {\hat P(\rho_\eps) - \hat P(\rho)}_{L^r_T(L^r(\Omega))}\\
&\to 0 .
\end{split}
\end{align}
Lastly, with a standard result on mollifiers we conclude
\begin{align}
\label{sc_PetaP}
\nabla P_{\eta} = \chi_\eta \ast (\nabla P(\rho)) \to \nabla P(\rho) \quad \text{in}\; L^2_T(L^2(M)) .
\end{align}
In the following we denote error terms by $\mathcal{E}_1 = \mathcal{E}_1(\eta,\eps)$ and $\mathcal{E}_2 = \mathcal{E}_2(\eta)$ if they satisfy
\begin{align*}
\limsup_{\eps\to 0} \;\left|\mathcal{E}_1(\eta,\eps)\right| \leq \mathcal{E}_2(\eta) \quad \text{and} \quad \lim_{\eta\to 0} \; \mathcal{E}_2(\eta) = 0 .
\end{align*}

\begin{lemma} We have
	\begin{align}
	\begin{split}
	\label{eqn_before_limit}
	&\int_0^T \int_{\Omega} \gamma \hat\rho_\eps \rho v^\eps_k \psi \nabla(\phi \ast_\eps \rho_\eps)  - 
	\rho v^\eps_k \psi \nabla P_{\eps,\eta} \diff x\diff t \\
	&\qquad=  \mu\int_0^T \int_{\Omega}  \rho \frac{\tilde u_\eps e_k}{\eps^2} \psi \diff x\diff t + \mathcal{E}_1(\eta,\eps) .
	\end{split}
	\end{align}
\end{lemma}

\begin{proof}
	Consider the test function $\rho v^\eps_k \psi$. This function is zero on $\partial \Omega_\eps$ as $v^\eps_k = 0$ on $\partial \Omega_\eps \setminus \partial \Omega$ and $\psi = 0$ near $\partial \Omega$. By using $\rho v^\eps_k \psi$ in the momentum equation \eqref{weak_mom} we get
	\begin{align*}
	&\int_0^T \int_{\Omega_\eps} \mu D u_\eps \dbdot D (\rho v^\eps_k \psi) \diff x \diff t \\
	&\qquad=
	\int_0^T \int_{\Omega_\eps} \gamma \rho_\eps (\rho v^\eps_k \psi) \nabla(\phi \ast_\eps \rho_\eps)  + 
	(P(\rho_\eps) -\xi \div u_\eps) \div (\rho v^\eps_k \psi) \diff x \diff t.
	\end{align*}
	We can extend the integral to $\Omega$ because $\tilde u$ and $v^\eps_k$ vanish on $\Omega \setminus \Omega_\eps$, so
	\begin{align}
	\label{eqn_lsx}
	\begin{split}
	&\int_0^T \int_{\Omega} \mu D\tilde u_\eps \dbdot D (\rho v^\eps_k \psi) \diff x \diff t \\
	&\qquad= 
	\int_0^T \int_{\Omega} \gamma \hat\rho_\eps \rho v^\eps_k \psi \nabla(\phi \ast_\eps \rho_\eps) + 
	(\hat P(\rho_\eps) -\xi \div \tilde u_\eps )\div (\rho v^\eps_k \psi) \diff x \diff t.
	\end{split}
	\end{align}
	Next, we estimate with \eqref{bounds_vi} and \eqref{bounds_hatP-Peta}
	\begin{align}
	\begin{split}
	\label{eqn_hatP-Peta}
	&\left|\int_0^T \int_{\Omega}(\hat P(\rho_\eps)-P_{\eps,\eta}) \div (\rho v^\eps_k \psi) \diff x \diff t \right| \\
	&\qquad=\left|\int_0^T \int_{\Omega}(\hat P(\rho_\eps)-P_{\eps,\eta}) (v^\eps_k \psi \nabla \rho + v^\eps_k \rho \nabla \psi)\diff x \diff t\right| \\
	&\qquad\leq\norm{\hat P(\rho_\eps) - P_{\eps,\eta}}_{L^2_T(L^2(M))} \norm{v^\eps_k}_{L^\infty(\Omega)}\norm{\rho}_{L^2_T(H^1(\Omega))} \\
	&\qquad\qquad\cdot \left(\norm{\psi}_{L^\infty_T(L^\infty)}+\norm{\nabla \psi}_{L^\infty_T(L^\infty)}\right) \\
	&\qquad\leq C(\eta + \eps) . 
	\end{split}
	\end{align}
	With the same argument and Lemma \ref{lemma_bounds1} we can also estimate
	\begin{align}
	\begin{split}
	\label{eqn_xi_div_ueps}
	&\left|\int_0^T \int_{\Omega}\xi \div \tilde u_\eps \div (\rho v^\eps_k \psi) \diff x \diff t \right| \\
	&\qquad=\left|\int_0^T \int_{\Omega}\xi \div \tilde u_\eps (v^\eps_k \psi \nabla \rho + v^\eps_k \rho \nabla \psi)\diff x \diff t \right| \end{split}\\
	&\qquad\leq \eps \xi \norm{\frac{u_\eps}{\eps}}_{L^2(H^1)} \norm{v^\eps_k}_{L^\infty}\norm{\rho}_{L^2_T(H^1)} \left(\norm{\psi}_{L^\infty_T(L^\infty)}+\norm{\nabla \psi}_{L^\infty_T(L^\infty)}\right) \notag\\
	&\qquad\leq C \eps . \notag
	\end{align}
	With this notation we can now use the estimates \eqref{eqn_hatP-Peta} and \eqref{eqn_xi_div_ueps} in equation \eqref{eqn_lsx} to obtain
	\begin{align}
	\begin{split}
	&\int_0^T \int_{\Omega} \mu D \tilde u_\eps \dbdot D (\rho v^\eps_k \psi) \diff x \diff t
	\\
	&\qquad= \int_0^T \int_{\Omega} \gamma \hat\rho_\eps \rho v^\eps_k \psi \nabla(\phi \ast_\eps \rho_\eps)  + 
	P_{\eps,\eta} \div (\rho v^\eps_k \psi) \diff x \diff t + \mathcal{E}_1(\eta,\eps) 
	\\
	&\qquad= \int_0^T \int_{\Omega} \gamma \hat\rho_\eps \rho v^\eps_k \psi \nabla(\phi \ast_\eps \rho_\eps)  - 
	\rho v^\eps_k \psi \nabla P_{\eps,\eta} \diff x \diff t + \mathcal{E}_1(\eta,\eps) .
	\label{ls1}
	\end{split}
	\end{align}
	Now we look at the differential equation satisfied by $v^\eps_k$: After rescaling \eqref{eqn_pdevi} gives us
	\begin{align*}
	-\eps^2 \Delta v^\eps_k + \eps \nabla q^\eps_k = e_k \quad \text{in}\;\; \Omega_{K,\eps} .
	\end{align*}
	Let us use the test function $\rho u_\eps \psi$. The test function is zero on $\partial \Omega_{K,\eps}$ as $u_\eps = 0$ on $\partial \Omega_\eps$ and $\psi = 0$ near $\partial \Omega$. We get
	\begin{align}
	\label{ls2x}
	\begin{split}
	&\int_0^T \int_{\Omega_{K,\eps}} D v^\eps_k \dbdot D (\rho u_\eps \psi) \diff x \diff t
	+ \frac{1}{\eps} \int_0^T \int_{\Omega_{K,\eps}} \rho u_\eps \psi \nabla q^\eps_k \diff x \diff t \\
	&\qquad= \int_0^T \int_{\Omega_{K,\eps}} \rho \frac{u_\eps}{\eps^2} \psi e_k \diff x \diff t.
	\end{split}
	\end{align}
	The second term vanishes in the limit $\eps \to 0$, $\eta \to 0$:
	\begin{align}
	\begin{split}
	\label{eqn_nablaqk}
	&\frac{1}{\eps} \int_0^T \int_{\Omega_{K,\eps}} \rho u_\eps \psi \nabla q^\eps_k \diff x \diff t\\
	&\qquad= \eps\int_0^T \int_{\Omega_{K,\eps}} \rho_\eps \frac{u_\eps}{\eps^2} \psi \nabla q^\eps_k  \diff x \diff t+ \mathcal{E}_1(\eta,\eps)
	\\
	&\qquad=  \eps\int_0^T \int_{\Omega_{K,\eps}} \rho_\eps \frac{u_\eps}{\eps^2} \nabla(\psi  q^\eps_k)  
	-  \rho_\eps \frac{u_\eps}{\eps^2}   q^\eps_k \nabla\psi \diff x \diff t
	+ \mathcal{E}_1(\eta,\eps)\\
	&\qquad= \mathcal{E}_1(\eta, \eps) .
	\end{split}
	\end{align}
	Here we have used the following estimates that are based on equation \eqref{bounds_qi}, Lemma \ref{lemma_bounds1}, Lemma \ref{lemma_sc}, and the use of \eqref{weak_conti}:
	
	\begin{align*}
	&\left|\eps\int_0^T \int_{\Omega_{K,\eps}} (\rho_\eps - \rho) \frac{u_\eps}{\eps^2} \psi \nabla q^\eps_k \diff x\diff t\right| \\
	&\qquad \leq \norm{\rho_\eps - \rho}_{L^2(L^2(\Omega_{K,\eps}))} \norm{\frac{u_\eps}{\eps^2}}_{L^2(L^2)}  \norm{\psi}_{L^\infty_T(L^\infty)} \norm{\eps\nabla q^\eps_k}_{L^\infty} ,\\
	&\left|\eps\int_0^T \int_{\Omega_{K,\eps}} \rho_\eps \frac{u_\eps}{\eps^2}   q^\eps_k \nabla\psi \diff x\diff t \right| \leq \eps \norm{\rho_\eps}_{L^2(L^2)} \norm{\frac{u_\eps}{\eps^2}}_{L^2(L^2)}  \norm{\nabla \psi}_{L^\infty_T(L^\infty)} \norm{q^\eps_k}_{L^\infty} , \\
	&\left|\eps\int_0^T \int_{\Omega_{K,\eps}} \rho_\eps \frac{u_\eps}{\eps^2} \nabla(\psi  q^\eps_k) \diff x\diff t \right| 
	= \left|-\eps\int_0^T \int_{\Omega_{K,\eps}} \rho_\eps \partial_t(\psi  q^\eps_k) \diff x\diff t \right| \\
	&\qquad\leq \eps \norm{\rho_\eps}_{L^2(L^2)} \norm{\partial_t \psi}_{L^\infty_T(L^\infty)} \norm{q^\eps_k}_{L^\infty} .
	\end{align*}
	Using \eqref{eqn_nablaqk} in \eqref{ls2x} and extending the integral from $\Omega_{K,\eps}$ to $\Omega$ we get
	\begin{align}
	\label{ls2}
	\int_0^T \int_{\Omega} D v^\eps_k \dbdot D(\rho \tilde u_\eps \psi) \diff x\diff t
	= \int_0^T \int_{\Omega}  \rho \frac{\tilde u_\eps e_k}{\eps^2} \psi \diff x\diff t + \mathcal{E}_1(\eta,\eps) .
	\end{align}
	We calculate
	\begin{align}
	\begin{split}
	&\int_0^T \int_\Omega D \tilde u_\eps \dbdot D (\rho v^\eps_k \psi) \diff x\diff t 
	- \int_0^T \int_{\Omega}  D v^\eps_k \dbdot D(\rho \tilde u_\eps \psi) \diff x\diff t\\
	&\qquad=\int_0^T \int_\Omega \big(
	\rho \psi D\tilde u_\eps \dbdot D v^\eps_k
	+ \psi  v^\eps_k      \cdot D \tilde u_\eps \nabla \rho 
	+\rho  v^\eps_k      \cdot D \tilde u_\eps \nabla \psi \\
	&\qquad\qquad \qquad\quad
	-\rho \psi D v^\eps_k \dbdot D\tilde u_\eps
	-\psi  \tilde u_\eps \cdot D v^\eps_k      \nabla \rho
	-\rho  \tilde u_\eps \cdot D v^\eps_k      \nabla \psi \big) \diff x\diff t \\
	&\qquad=\int_0^T \int_\Omega
	\left( v^\eps_k \cdot D \tilde u_\eps -  \tilde u_\eps \cdot D v^\eps_k\right) 
	\left(\psi \nabla \rho + \rho\nabla \psi \right) \diff x\diff t\\
	&\qquad=\mathcal{E}_1(\eta,\eps) ,
	\label{ls3}
	\end{split}
	\end{align}
	where we have used equation \eqref{bounds_vi}, Lemma \ref{lemma_bounds1} and the estimate
	\begin{align*}
	& \left|\int_0^T \int_\Omega
	\left( v^\eps_k \cdot D \tilde u_\eps -  \tilde u_\eps \cdot D v^\eps_k\right) 
	\left(\psi \nabla \rho + \rho\nabla \psi \right) \diff x\diff t \right|\\
	&\qquad \leq \eps \left(\norm{\frac{\tilde u_\eps}{\eps^2}}_{L^2(L^2)} \norm{\eps D v^\eps_k}_{L^\infty}+
	\norm{\frac{D \tilde u_\eps}{\eps}}_{L^2(L^2)} \norm{v^\eps_k}_{L^\infty} \right) \\
	&\qquad \quad \cdot 
	\left(\norm{\psi}_{L^\infty_T(L^\infty)}\norm{\nabla \rho}_{L^2(L^2)}
	+ \norm{\nabla \psi}_{L^\infty_T(L^\infty)}\norm{\rho}_{L^2(L^2)}\right).
	\end{align*}
	Now we can combine \eqref{ls1}, \eqref{ls2} and \eqref{ls3} to get the assertion of the lemma.
\end{proof}
To take the limit $\eps\to 0$ in equation \eqref{eqn_before_limit}, let us first summarize the convergences as proven up to now:
\begin{corollary} \label{cor_conv}
	We have
	\begin{align*}
	\hat \rho_\eps &\to \rho &&\text{in} \; L^2_T(L^2(\Omega)) ,
	\\
	v^\eps_k &\overset{\ast}\weakto \int_{\Ys} \tilde v_k = \bar A e_k  &&\text{in} \; L^\infty(\Omega)=L^1(\Omega)^\ast 
	,\\
	\qquad \nabla(\phi \ast_\eps \rho_\eps) &\to \theta \nabla(\phi \ast_0 \rho)  &&\text{in} \; L^r_T(L^r(\Omega))\; \text{for any}\; 1\leq r < \infty
	,\\
	\nabla P_{\eps,\eta} &\to  \nabla P_{\eta} &&\text{in} \; L^r_T(L^r(M))\; \text{for any}\; 1\leq r < 2
	,\\
	\frac{\tilde u_\eps}{\eps^2} &\weakto u &&\text{in}\; L^2_T(L^2(\Omega)) .
	\end{align*}
\end{corollary}
\begin{proof}
	See Corollary \ref{cor_wc}, Lemma \ref{lemma_sc}, and equation \eqref{sc_Peta}. The weak-$\ast$ convergence of $v^\eps_k$ defined in \eqref{definition_vq} is a standard result.
\end{proof}
Recall that $\bar A$ is the permeability matrix as defined in \eqref{eqn_barA}. Also recall that from \eqref{wc_q} and Lemma \ref{lemma_sc} we have $\rho \in L^{2\beta}_T(L^{2\beta})$. In the limit $\eps \to 0$ equation \eqref{eqn_before_limit} gives
\begin{align*}
\int_0^T \int_{\Omega} \gamma \rho^2 \psi \theta \nabla(\phi \ast_0 \rho) \bar A e_k - 
\rho \psi \nabla P_{\eta} \bar A e_k \diff x \diff t=  \mu\int_0^T \int_{\Omega}  \rho u \psi e_k \diff x \diff t + \mathcal{E}_2(\eta) .
\end{align*}
From \eqref{sc_PetaP} we know that $\nabla P_{\eta} \to \nabla P(\rho)$ in $L^2_T(L^2(M))$, so we can pass to the limit $\eta\to 0$, concluding
\begin{align*}
\int_0^T \int_{\Omega} \bar A e_k \left(\gamma \theta \rho^2 \nabla(\phi \ast_0 \rho) - 
\rho \nabla P(\rho)\right) \psi \diff x \diff t =  \mu\int_0^T \int_{\Omega}  \rho u \psi e_k \diff x \diff t .
\end{align*}
Here $\psi \in C^\infty_C((0,T)\times \Omega)$ is an arbitrary test function. Collecting the results for $k=1,\ldots, N$ we get
\begin{align*}
\bar A \left(\gamma \theta \rho^2 \nabla(\phi \ast_0 \rho) - 
\rho \nabla P(\rho)\right) = \mu \rho u .
\end{align*}
Now we can use the identity $P(\rho) = p(\rho) + \frac{\gamma}{2} \rho^2$ to restate this result in terms of $p$:
\begin{align}
\label{Result2x}
\bar A \left(\gamma \theta \rho^2 \nabla(\phi\ast_0\rho - \rho) - \rho \nabla \left(p(\rho)+\frac{\gamma (1-\theta)}{2}\rho^2 \right)\right) = \mu \rho u .
\end{align}
This directly implies \eqref{eqn_result_2x}.

We now are able to infer the weak formulation of \eqref{Result_CH}, \eqref{Result_CH_IVBP}. For this we use Lemma \ref{lemma_masmoudi1} to write
\begin{align*}
-\int_0^T \int_\Omega \tilde \rho_\eps \partial_t \psi \diff x \diff t 
- \int_0^T \int_\Omega \hat \rho_\eps \frac{\tilde u_\eps}{\eps^2} \nabla \psi \diff x \diff t 
= \int_\Omega \tilde \rho_{0,\eps} \psi(t=0) \diff x,
\end{align*}
with a test function $\psi \in C^\infty_C((-\infty,T)\times\RR^N)$. To pass to the limit $\eps \to 0$ we recall that the initial conditions satisfy $\hat \rho_{0,\eps} \weakto \rho_0$ in $L^1(\Omega)$ and with Lemma \ref{lemma_extension} we conclude $\tilde \rho_{0,\eps} \weakto \theta\rho_0$. The other convergences have been discussed in Corollary \ref{cor_conv}, we conclude
\begin{align*}
-\int_0^T \int_\Omega \theta \rho \partial_t \psi \diff x \diff t
- \int_0^T \int_\Omega \rho u \nabla \psi \diff x \diff t
= \int_\Omega \theta \rho_{0} \psi(t=0) \diff x.
\end{align*}
With this we have shown Theorem \ref{theorem}.

\appendix
\section{Appendix\label{chapter_MathAppendix}}
This appendix will provide the technical details needed in the proof of Theorem \ref{theorem}. 
First we  introduce a  restriction operator that is dual to the mean value extension in Definition \ref{def_extension}.  Similar operators have been constructed by Tartar \cite{tartar}, in particular we refer to the construction of  Lipton\&Avellaneda \cite{av}.
\begin{lemma}
	\label{lemmma_restriction}
	There exists a restriction operator $R_\eps: H^1_0(\Omega)^N  \to H^1(\Omega_\eps)^N$ such that 
	\begin{enumerate}
		\item $R_\eps u = 0$ on $\partial \Omega$ and $(R_\eps u) \cdot n = 0$ on $\partial \Omega_\eps$,
		\item $R_\eps$ satisfies
		\begin{align}
		\label{eqn_lemma_res_1}
		\int_{\Omega_\eps} u \div R_\eps v \diff x= \int_\Omega \hat u \div v \diff x \qquad \forall u \in L^2(\Omega_\eps),v\in H^1_0(\Omega)^N,
		\end{align}
		\item there exists a constant $C>0$ such that for all $v\in H^1_0(\Omega)^N$
		\begin{align}
		\label{eqn_lemma_res_2}
		\norm{R_\eps v}_{L^2(\Omega_\eps)^N} + \eps \norm{D R_\eps v}_{L^2(\Omega_\eps)^{N\times N}} \leq C \left(\norm{v}_{L^2(\Omega)^N}+\eps \norm{D v}_{L^2(\Omega)^{N\times N}}\right).
		\end{align}
	\end{enumerate}
\end{lemma}

\begin{proof}
	It is shown in Lemma 2.2 of \cite{av} that given $u\in H^1(\Ys)^N$ there exist solutions $v\in H^1(\Ys_{r\setminus s})^N$ and $q\in L^2(\Ys_{r\setminus s})/\RR$ of
	\begin{align*}
	&\Delta v = \Delta u - \nabla q ,\\
	&\div v = \div u + \frac{1}{|\Ys_{r\setminus s}|} \int_{\partial \Ys_s} u\cdot n \diff s ,\\
	&v = u \quad \text{on}\;\; \partial \Ys_r ,\\
	&v\cdot n = 0,\; v\cdot \tau = u\cdot \tau \quad \text{on}\;\; \partial \Ys_s .
	\end{align*}
	Here $n\in \RR^N$ is the outward normal unit vector of $\partial \Ys_s$ and $\tau\in \RR^N$ is any tangent vector of $\partial \Ys_s$. The region $\Ys_{r\setminus s}$ is defined in Section \ref{chapter_domain}. Note that we used the same region $\Ys_{r\setminus s}$ for the mean--value extension in Definition \ref{def_extension}. We define an operator $R: H^1(\Ys)^N \to H^1(\Ys_f)^N$ by
	\[
	Ru = \begin{cases}
	u & \text{in}\;\; \Ys_f\setminus \Ys_{r\setminus s} ,\\
	v & \text{in}\;\; \Ys_{r\setminus s} .
	\end{cases}
	\]
	To show \eqref{eqn_lemma_res_1} we need to do a calculation: Let us have $v\in H^1(\Ys)^N$ and $u\in L^2(\Ys_f)$. Let $\hat u$ be the extension of $u$ to $\Ys$ by its mean value in $\Ys_{r\setminus s}$. Then
	\begin{align}
	\begin{split}
	\label{eqn_udivRv}
	&\int_{\Ys_f} u \div Rv \diff x - \int_{\Ys} \hat u \div v \diff x
	= \int_{\Ys_{r\setminus s}}\left(u \div Rv-\hat u \div v\right) \diff x - \int_{\Ys_s} \hat u \div v \diff x\\
	&\qquad= \int_{\Ys_{r\setminus s}}\left(u \left(\div v + \frac{1}{|\Ys_{r\setminus s}|} \int_{\partial \Ys_s} v\cdot n \diff s\right) - u \div v\right) \diff x \\
	&\qquad\quad - \left(\frac{1}{|\Ys_{r\setminus s}|} \int_{\Ys_{r\setminus s}} u \diff x \right)\int_{\Ys_s} \div v \diff x\\
	&\qquad= \left(\frac{1}{|\Ys_{r\setminus s}|} \int_{\Ys_{r\setminus s}} u \diff x \right) \left(\int_{\partial \Ys_s} v\cdot n \diff s - \int_{\Ys_s} \div v\diff x \right) \\
	&\qquad= 0 .
	\end{split}
	\end{align}
	We define $R_\eps$ by applying a rescaled $R$ on every cell in the interior of $\Omega_\eps$. To be precise, let for $\eps>0$ and $k\in K_\eps$ the affine transformation $T_{k,\eps}: \Ys \to \eps(\Ys + k)$ be given through $T_{k,\eps} x = \eps (x+k)$ and set for $u\in H^1_0(\Omega)^N$
	\begin{align*}
	R_\eps u = \begin{cases}
	u & \text{in}\;\; \Omega_\eps\setminus \Omega_{K,\eps} ,\\
	R [u \circ T_{k,\eps}]\circ T_{k,\eps}^{-1} & \text{in}\;\; \eps(\Ys_f+k),\;\; k\in K_\eps .
	\end{cases}
	\end{align*}
	Then we easily get \eqref{eqn_lemma_res_2} analogously to \cite{av}, Lemma 2.1. Furthermore, because the rescaling of $R$ and of the mean value extension used in \eqref{eqn_udivRv} happens in the same way, \eqref{eqn_udivRv} implies \eqref{eqn_lemma_res_1}.
\end{proof}


The next lemma concerns the convolution defined in \eqref{abbrev_convolution}. A nice property of this convolution is that we get a convergence of the form $\nabla(\phi \ast_\eps f) \to \theta\nabla(\phi \ast_0 f)$.

\begin{lemma}
	\label{lemma_conv1}
	Let  $\phi$ be an interaction kernel satisfying \eqref{model_kernel}. Let $f \in L^1_T(L^1(\Omega))$.
	\begin{enumerate}
		\item \label{lemma_conv1_i3} The function $\phi \ast_\eps f$ is smooth in $\Omega$ with
		\begin{align*}
		\norm{\nabla(\phi \ast_\eps f)}_{L^\infty_T(L^\infty(\Omega))} \leq \norm{\nabla\phi}_{L^\infty}\norm{f}_{L^1_T(L^1(\Omega))} + \norm{\nabla \phi}_{L^1} \rho_s .
		\end{align*}
		
		\item \label{lemma_conv1_i4} We have for all $r\in [1, \infty)$
		\begin{align*}
		\nabla(\phi \ast_\eps f) \to \theta\nabla(\phi \ast_0 f) \quad \text{in} \quad L^r_T(L^r(\Omega)) .
		\end{align*}
		
	\end{enumerate}
\end{lemma}

\begin{proof}
	The first assertion follows trivially from estimating
	\begin{align*}
	\nabla (\phi \ast_\eps f)(t,x) = \int_{\Omega_\eps}\nabla \phi(x-y)f(t,y)\diff y + \rho_s \int_{\RR^N\setminus \Omega_\eps}\nabla \phi(x-y)\diff y.
	\end{align*}
	
	For the second assertion let $\psi \in C^\infty_C(\RR^N)$. Then $\psi$ is Lipschitz continuous with some Lipschitz constant $L$. Splitting the domain $\Omega_\eps$ into the cells, we have
	\begin{align}
	\label{conv11}
	\int_{\Omega_\eps} \psi(y) \diff y
	&= \sum_{k\in K_\eps} \int_{\eps(\Ys_f + k)} \psi(y) \diff y +   \int_{\Omega\setminus\Omega_K} \psi(y) \diff y
	\end{align}
	and
	\begin{align}
	\label{conv12}
	\int_{\Omega} \psi(y) \diff y 
	&= \sum_{k\in K_\eps} \int_{\eps(\Ys + k)} \psi(y) \diff y +   \int_{\Omega\setminus\Omega_K} \psi(y) \diff y .
	\end{align}
	On a cell we can exchange values $\psi(y)$ by the value of $\psi$ at a single point $\eps k$ with an error bounded by $L\sqrt{N}\eps$, i.e.,
	\begin{align*}
	&\left| \int_{\eps(\Ys_f + k)} \psi(y) \diff y - \theta \int_{\eps(\Ys + k)} \psi(y) \diff y\right| \\
	&\qquad\leq \left| \int_{\eps(\Ys_f + k)} \psi(\eps k) \diff y - \theta \int_{\eps(\Ys + k)} \psi(\eps k) \diff y\right| \\
	&\qquad\quad+\int_{\eps(\Ys_f + k)} L\sqrt{N}\eps \diff y + \theta \int_{\eps(\Ys + k)} L\sqrt{N}\eps \diff y \\
	&\qquad= \eps^N |\Ys_f|L\sqrt{N}\eps + \theta \eps^N |\Ys|L\sqrt{N}\eps \\
	&\qquad= 2 \theta L\sqrt{N}\eps^{N+1} .
	\end{align*}
	Now we take $\eqref{conv11}-\theta \cdot \eqref{conv12}$ and apply this estimate on every cell $\eps(\Ys+k)$, $k\in K_\eps$:
	\begin{align}
	\begin{split}
	\label{conv13}
	\left|\int_{\Omega_\eps} \psi(y) \diff y - \theta \int_{\Omega} \psi(y) \diff y\right| 
	&\leq \sum_{k\in K_\eps} 2 \theta L\sqrt{N}\eps^{N+1} +   (1-\theta)\int_{\Omega\setminus\Omega_K} \psi(y) \diff y \\
	&\leq 2 |\Omega|\theta L\sqrt{N}\eps +   (1-\theta)\big|\Omega\setminus\Omega_K\big| \norm{\psi}_{L^\infty(\Omega)} .
	\end{split}
	\end{align}
	As $\Omega$ has a smooth boundary, $\Omega_K$ is a good approximation to $\Omega$ in the sense that $|\Omega\setminus\Omega_K\big| \to 0$ as $\eps \to 0$.

	Next we will show pointwise convergence of $\nabla(\phi \ast_\eps f)$. For this fix some $(t,x)\in (0,T)\times \Omega$ and define $g\in L^1(\Omega)$ by $g(y) = \nabla \phi(x-y)(f(t,y)-\rho_s)$. This is motivated by
	\begin{align}
	\begin{split}
	\label{conv14}
	\nabla(\phi \ast_\eps f)(t,x) &= \nabla\left(\int_{\Omega_\eps} \nabla\phi(x-y)f(t,y) \diff y+ \int_{\RR^N \setminus \Omega_\eps}\phi(x-y)\rho_s\diff y\right) \\
	&=\int_{\Omega_\eps} \nabla\phi(x-y)(f(t,y)-\rho_s)\diff y = \int_{\Omega_\eps} g(y)\diff y
	\end{split}
	\end{align}
	Fix $\delta > 0$ and choose a $\psi \in C^\infty_C(\Omega)$ with $\norm{g-\psi}_{L^1(\Omega)} \leq \delta$. Now we can use \eqref{conv13} to conclude that for $\eps$ small enough
	\begin{align*}
	\left|\int_{\Omega_\eps} g \diff y- \theta \int_{\Omega} g \diff y\right|
	\leq 2 \delta +\left|\int_{\Omega_\eps} \psi \diff y- \theta \int_{\Omega} \psi \diff y\right| \leq 3 \delta .
	\end{align*}
	Because $\delta>0$ was arbitrary we conclude $\int_{\Omega_\eps} g \diff y \to \theta \int_{\Omega} g \diff y$ as $\eps \to 0$. With the identity \eqref{conv14} this means pointwise convergence $\nabla(\phi \ast_\eps f) \to \theta\nabla (\phi \ast_0 f)$.
	
	Using the $L^\infty$-bounds from the first assertion we can apply the dominated convergence theorem for $L^p$-functions, we have $\nabla(\phi \ast_\eps f) \to \theta\nabla (\phi \ast_0 f)$ in $L^r_T(L^r(\Omega))$ for all $1\leq r < \infty$.
\end{proof}


The requirements in Definition \ref{def_pressure} give rise to the following property:
\begin{lemma}
	\label{lemma_pressure}
	Let $P(r) = p(r) + \frac{\gamma}{2} r^2$ be an admissible  generalized pressure function with constants $\alpha>0$, $\beta \geq 2$, $c>0$ as in Definition \ref{def_pressure}. Let the energy function $W$ be given through \eqref{model_W}. In the case $\beta > 2$ we have
	\begin{align*}
	\lim_{r\to\infty} \frac{W(r)}{P(r)} = \frac{1}{\beta-1}.
	\end{align*}
\end{lemma} 

\begin{proof}
	Recall that by requirement \ref{def_pressure_i5} of Definition \ref{def_pressure}
	\begin{align*}
	\lim_{r\to\infty} \frac{P'(r)}{r^{\beta-1}} = c.
	\end{align*}
	In the case $\beta > 2$ this behaviour can not originate from the term $\frac{\gamma}{2} r^2$, we can conclude
	\begin{align*}
	\lim_{r\to\infty} \frac{p'(r)}{r^{\beta-1}} = c.
	\end{align*}
	Recall the definition of the energy function $W$, that is $p'(r) = rW''(r)$. Using the L'Hospital rule twice we finally get
	\begin{align*}
	\lim_{r\to\infty} \frac{W(r)}{r^{\beta}} = \frac{c}{\beta(\beta-1)}.
	\end{align*}
	We conclude with requirement \ref{def_pressure_i5} of Definition \ref{def_pressure}
	\begin{align*}
	\lim_{r\to\infty} \frac{W(r)}{P(r)} 
	=  \left(\lim_{r\to\infty} \frac{W(r)}{r^{\beta}} \right)\left(\lim_{r\to\infty} \frac{P(r)}{r^{\beta}} \right)^{-1} = \frac{c}{\beta(\beta-1)} \frac{\beta}{c} = \frac{1}{\beta-1} .
	\end{align*}
\end{proof}


The next lemma generalizes the following observation about weak limits: If we have $\phi_n \weakto \phi$ in $L^2$ and $\phi_n^2 \weakto \psi$ in $L^1$ then $\psi \geq \phi^2$. We need this lemma to compare weak limits in Section \ref{chapter_sc}.

\begin{lemma}
	\label{lemma_weak_leq}
	Let $D$ be a bounded domain in $\RR^N$. Let $f: \RR\to\RR$ be a convex function with 
	\begin{align*}
	f\geq 0,\quad f(x) = f(-x), \quad \lim_{x\to\infty} \frac{f'(x)}{x^{b-1}} = c\in[0,\infty)
	\end{align*}  
	for some constant $b \geq 1$. Fix some constant $a$ with 
	\begin{align*}
	2\leq a < \infty\quad\text{and}\quad 2(b-1) \leq a.
	\end{align*}
	Let $\phi_n \in L^a(D)$ be a sequence with weak $L^a$-limit $\phi \in L^a(D)$. Furthermore let $f(\phi_n)$ have the weak $L^1$-limit $\psi\in L^1(D)$. Then $\psi \geq f(\phi)$.
\end{lemma}
\begin{proof}
	We have $a \geq b$: If $b\leq 2$ this is obvious, otherwise $a \geq 2(b-1) \geq b$.
	
	As $\frac{f'(x)}{x^{b-1}} \to c$ we can find two constants $C_1$, $C_2$ such that $|f'(x)| \leq C_1 + C_2 |x|^{b-1}$. By  L'Hospital we have $\frac{f(x)}{x^{b}} \to c/b$. With the same argument as before we find two constants $C_3$, $C_4$ such that $f(x) \leq C_3 + C_4 |x|^b$. Because $D$ is a bounded domain, this implies $f(\phi_n), f(\phi) \in L^{a / b}(D) \subset L^1(D)$.
	
	Now assume $\psi < f(\phi)$ on some set $M\subset D$ with $|M|>0$. We will show that the functional $g: L^a(D)\to \RR,\; \eta \mapsto \norm{f(\eta)}_{L^1(M)}$ is continuous and convex. For this we calculate
	\begin{align*}
	&|g(\eta_1)-g(\eta_2)| = \left|\int_M f(\eta_1) - f(\eta_2) \diff x \right| \\
	&\qquad\leq \int_M  \left| \eta_1(x) - \eta_2(x)\right| \max_{y\in [\eta_1(x),\eta_2(x)]} |f'(y)| \diff x \\
	&\qquad\leq \int_M  \left| \eta_1(x) - \eta_2(x)\right| \left(C_1 + C_2 |\eta_1(x)|^{b-1} + C_2|\eta_2(x)|^{b-1}\right) \diff x \\
	&\qquad\leq \norm{\eta_1 - \eta_2}_{L^2(M)} C \left(1 + \norm{|\eta_1|^{b-1}}_{L^2(M)} +\norm{|\eta_2|^{b-1}}_{L^2(M)}\right) .
	\end{align*}
	For fixed $\eta_1$ and $\eta_2$ approaching $\eta_1$ in $L^a(D)$ the first term converges to zero (as $a \geq 2$) and the second term is bounded (as $a \geq 2(b-1)$).  We have shown the continuity of $g$. Furthermore for $\eta_1,\eta_2 \in L^a(D)$, $\lambda \in [0,1]$ we have
	\begin{align*}
	g(\lambda \eta_1 + (1-\lambda) \eta_2) &= \norm{f(\lambda \eta_1 + (1-\lambda) \eta_2)}_{L^1(M)} \\ 
	&\leq \norm{\lambda f(\eta_1) + (1-\lambda) f(\eta_2)}_{L^1(M)} \\
	&\leq \lambda \norm{f(\eta_1)}_{L^1(M)} + (1-\lambda) \norm{f(\eta_2)}_{L^1(M)} \\
	&= \lambda g(\eta_1) + (1-\lambda) g(\eta_2) .
	\end{align*}
	As $g$ is continuous and convex, it is weakly lower semicontinuous, and therefore $g(\phi)\leq \liminf_{n\to\infty} g(\phi_n) $. But this means
	\begin{align*}
	\norm{f(\phi)}_{L^1(M)} > \norm{\psi}_{L^1(M)} = \lim_{n\to\infty} \norm{f(\phi_n)}_{L^1(M)} \geq \norm{f(\phi)}_{L^1(M)} .
	\end{align*}
	This is a contradiction to our assumption $\psi < f(\phi)$, so we have $\psi \geq f(\phi)$.
\end{proof}

We want to estimate $\norm{f-\chi_\eta \ast f}_{L^2(D)}$ where $\chi_\eta$ is a mollifier function. It is well known that for $f\in L^2(D)$ this norm converges to zero. But we are interested in a bound that converges to zero with a certain rate. This is possible as long as we have higher regularity:
\begin{lemma}
	\label{lemma_mol}
	Let $f\in H^1(D)$ for some domain $D\subset \RR^N$. Let $\chi$ be a standard mollifier function, that is $\chi \in C^\infty_C(B(1,0))$, $\chi \geq 0$, $\int_{\RR^N} \chi = 1$, and for $\eta>0$ let $\chi_\eta(x) = \eta^{-N}\chi(x/\eta)$. Then there exists a constant $C>0$ such that
	\begin{align*}
	\norm{f-\chi_\eta \ast f}_{L^2(D_\eta)} \leq C \eta \norm{f}_{H^1(D)} ,
	\end{align*}
	where $D_\eta = \set{x\in D: \dist(x,\partial D) \geq \eta}$.
\end{lemma}

\begin{proof}
	We calculate with Jensen's inequality
	\begin{align*}
	\norm{f-\chi_\eta \ast f}^2_{L^2(D_\eta)} &= \int_{D_\eta}\left(\int_{B(\eta,0)} \chi_\eta(y) (f(x)-f(x-y)) \diff y \right)^2 \diff x \\
	&\leq \int_{D_\eta}\int_{B(\eta,0)} \chi_\eta(y) (f(x)-f(x-y))^2 \diff y \diff x .
	\end{align*}
	Using the notation
	\begin{align*}
	D^y f(x) := \frac{f(x)-f(x-y)}{|y|}
	\end{align*}
	we can write
	\begin{align*}
	\norm{f-\chi_\eta \ast f}^2_{L^2(D_\eta)} &\leq \int_{D_\eta}\int_{B(\eta,0)} \chi_\eta(y) |y|^2 (D^y f(x))^2 \diff y \diff x \\
	&= \int_{B(\eta,0)} \chi_\eta(y) |y|^2 \norm{D^y f}_{L^2(D_\eta)}^2 \diff y .
	\end{align*}
	By \cite{evans}, Chapter 5, Theorem 3, there exists a $C>0$ such that for all $y\in B(\eta,0)$
	\begin{align*}
	\norm{D^y f}_{L^2(D_\eta)} \leq C \norm{f}_{H^1(D)} .
	\end{align*}
	We conclude
	\begin{align*}
	\norm{f-\chi_\eta \ast f}^2_{L^2(D_\eta)} &\leq \int_{B(\eta,0)} \chi_\eta(y) |y|^2 C^2 \norm{f}_{H^1(D)}^2 \diff y \\
	&\leq  \eta^2 C^2 \norm{f}_{H^1(D)}^2 \int_{B(\eta,0)} \chi_\eta(y) \diff y .
	\end{align*}
\end{proof}


\bibliographystyle{siamplain}
\bibliography{references}
\end{document}